\documentclass{amsart}
\usepackage{latexsym}
\usepackage{amsxtra}
\usepackage{amssymb}
\usepackage{amsthm}
\usepackage{xspace}

\usepackage{euscript}
\input xy
\xyoption{all} \CompileMatrices

\newtheorem{theorem}{Theorem}[section]
\newtheorem{proposition}[theorem]{Proposition}

\newtheorem{corollary}[theorem]{Corollary}
\newtheorem{remark}[theorem]{Remark}
\newtheorem{lemma}[theorem]{Lemma}

\newtheorem{algorithm}[theorem]{Algorithm}

 \numberwithin{equation}{section} 

\begin{document}

\newcommand{\ZZ}{\mathbb Z}
\newcommand{\CC}{\mathbb C}
\newcommand{\RR}{\mathbb R}
\newcommand{\QQ}{\mathbb Q}
\newcommand{\FF}{\mathbb F}
\newcommand{\OO}{\mathcal O}
\newcommand{\PP}{\mathcal P}
\newcommand{\mm}{\mathfrak m}
\newcommand{\nn}{\mathfrak n}
\newcommand{\ff}{\mathfrak f}
\newcommand{\aaa}{\mathfrak a}
\newcommand{\bbb}{\mathfrak b}
\newcommand{\ppp}{\mathfrak p}
\newcommand{\qqq}{\mathfrak q}
\newcommand{\PPP}{\mathfrak P}
\newcommand{\PPM}{\mathfrak M}
\newcommand{\AAA}{\mathfrak A}
\newcommand{\qq}{\mathfrak Q}
\newcommand{\Mod}{\operatorname{mod}}
\newcommand{\im}{\operatorname{im}}
\newcommand{\RRR}{\widehat{R}}

\title{Factoring formal power series over principal ideal domains}
\author{Jesse Elliott }  \address{California
State University, Channel Islands\\ One University Drive \\ Camarillo, California 93012}
\email{jesse.elliott@csuci.edu} \date{\today}

\begin{abstract}
We provide an irreducibility test and factoring algorithm (with some qualifications) for formal power series in the unique factorization domain $R[[X]]$, where $R$ is any principal ideal domain.  We also classify all integral domains arising as quotient rings of $R[[X]]$.  Our main tool is a generalization of the $p$-adic Weierstrass preparation theorem to the context of complete filtered commutative rings.
\end{abstract}

\maketitle

\section{Introduction}\label{intro}

All rings and algebras in this paper are assumed commutative with identity.    For any ring $R$, if $f$ is a polynomial in $R[X]$ or a formal power series in $R[[X]]$, then we let $f_i$, or $(f)_i$ when necessary, denote the coefficient of $X^i$ in $f$.  A homomorphism of $R$-algebras will be called an {\it $R$-homomorphism}.

It is well-known that formal power series rings exhibit pathologies that their polynomial ring counterparts do not \cite{brewer}.  For example, if $R$ is a ring of finite Krull dimension $n = \dim R$, then $n+1 \leq \dim R[X] \leq 2n+1$, while $\dim R[[X]]$ may be infinite, even if $\dim R = 0$.  Also, if $R$ is a unique factorization domain (UFD), then $R[X]$ is also a UFD, but $R[[X]]$ need not be.  In fact it is unknown whether or not $R[[X]]$ is a UFD if $R$ is the ring of polynomials in a countably infinite number of variables over a field $K$ \cite{gilmer}.

Interesting problems concerning $R[[X]]$ arise even for $R = \ZZ$.  For example, there is no known irreducibility criterion for the elements of $\ZZ[[X]]$ \cite{gil1} \cite{gil2}.   As noted in \cite{gil1} there are polynomials that are reducible in $\ZZ[X]$ while irreducible in $\ZZ[[X]]$, such as $(1+X)(2+X)$, and likewise there are polynomials that are irreducible in $\ZZ[X]$ while reducible in $\ZZ[[X]]$, such as $6+X$.  (See Propositions \ref{simplefactor1} and \ref{simplefactor}.)  The articles \cite{gil1} \cite{gil2} \cite{gil3} provide the following: (1) sufficient conditions for a power series $f \in \ZZ[[X]]$ to be irreducible; (2) sufficient conditions for $f$ to be reducible, along with factorization algorithms in those cases; and (3) necessary and sufficient conditions for $f$ to be irreducible in $\ZZ[[X]]$ if $f$ is a polynomial in $\ZZ[X]$ of degree at most $3$.  In this paper we provide an irreducibility test and factoring algorithm (with some qualifications) for formal power series in the UFD $R[[X]]$ for any principal ideal domain (PID) $R$.

Let $R$ be a ring.   A nonzero element $a$ of $R$ is said to be {\it prime (in $R$)} if $(a)$ is a prime ideal of $R$.  For any ideal $\aaa$ of $R$ we denote by $\RRR_\aaa$ the $\aaa$-adic completion $\displaystyle \varprojlim R/\aaa^n$ of $R$.  The following irreducibility criterion is proved in Section \ref{sec:3}.  

\begin{theorem}\label{irredtest2}
Let $R$ be a PID.    A polynomial $f \in R[X]$ is irreducible in $R[[X]]$ if and only if one of the following conditions holds.
\begin{enumerate} 
\item $f_0 = 0$ and $f_1$ is a unit in $R$.
\item $f_0$ is associate to a power of some prime $\pi \in R$, and if $f = gh$ with $g,h \in \RRR_{(\pi)}[X]$, then either $g_0$ or $h_0$ is a unit in $\RRR_{(\pi)}$.
\end{enumerate}
\end{theorem}

Theorem \ref{irredtest2} shows in particular that the irreducibility in $\ZZ[[X]]$ of a polynomial $f \in \ZZ[X]$ not associate to $X$ depends precisely on (1) how $f_0$ factors in $\ZZ$---the constant term $f_0$ must be $\pm p^k$ for some prime $p$; and (2)  how $f$ factors in the polynomial ring $\ZZ_p[X]$, where $\ZZ_p$ denotes the ring $\widehat{\ZZ}_{(p)}$ of $p$-adic integers---$f$ must have a unique irreducible factor, up to associate, with constant term in $p\ZZ_p$.  That the $p$-adic numbers are involved in this problem is suggested by several results in \cite{gil2} \cite{gil3}, but the relationship is not surprising since $f \in \ZZ[[X]]$ is irreducible if and only if the quotient ring $\ZZ[[X]]/(f)$ is an integral domain, and the ring $\ZZ_p \cong \ZZ[[X]]/(p-X)$ arises as such a quotient ring.  Similarly, for example, the quotient ring  $\ZZ[[X]]/(p-X^2)$ is isomorphic to $\ZZ_p[\sqrt{p}]$.  This follows from Theorem \ref{irredtest} below, but to verify the isomorphism directly one may show that the obvious ring homomorphism from $\ZZ[[X]]$ to $\CC_p$ sending $X$ to $\sqrt{p}$, where $\CC_p$ denotes the field of complex $p$-adic numbers, has the appropriate kernel and image.  Generalizing these examples, \cite[Theorem 3.1.1]{mcd} shows that the integral domains arising as quotient rings of $\ZZ[[X]]$ are, up to isomorphism, precisely the following: $\ZZ[[X]]$, $\ZZ$, $\FF_p[[X]]$, $\FF_p$, and $\ZZ_p[\alpha]$, where $p$ is any prime and $\alpha$ is any element of $\overline{\QQ_p}$ with $v(\alpha) > 0$, where $v$ is the unique valuation on $\CC_p$ extending the $p$-adic valuation on $\QQ_p$.  This result generalizes as follows.  Define the {\it absolute integral closure} $R^+$ of an integral domain $R$ to be the integral closure of $R$ in an algebraic closure of its fraction field.  It is known that $R^+$ is local if $R$ is a Henselian local domain \cite{art}.

\begin{theorem}\label{quotientrings}
Let $R$ be a PID.  The integral domains arising as quotient rings of $R[[X]]$ are, up to isomorphism, precisely the following: $R[[X]]$, $R$, $(R/\mm)[[X]]$, $R/\mm$, and $\RRR_\mm[\alpha]$, where $\mm$ is any maximal ideal of $R$ and $\alpha$ is any element of the unique maximal ideal of the absolute integral closure of $\RRR_\mm$.
\end{theorem}

The main tool in our proofs of the results above and indeed in our whole investigation is a generalization of the $p$-adic Weierstrass preparation theorem to the setting of complete filtered rings.  Let $R$ be a ring.  A {\it filtration} of $R$ is an infinite descending sequence $\ff = (\ff_0, \ff_1, \ff_2, \ldots)$ of ideals of $R$.   Let $\RRR_{\ff} = \varprojlim R/\ff_i$ denote the completion of $R$ with respect to $\ff$.  We say that $R$ is {\it complete with respect to $\ff$} if the natural homomorphism $R \longrightarrow \RRR_{\ff}$ is an isomorphism.   This is equivalent to saying that $R$ is complete and Hausdorff in the $\ff$-topology.   The completion $\RRR_{\ff}$ of $R$ is complete with respect to the filtration $\widehat{\ff} =  (\widehat{\ff_0}, \widehat{\ff_1}, \widehat{\ff_2}, \ldots)$, where $\widehat{I}$ for any ideal $I$ of a topological ring $S$ denotes the closure of $I$ in $S$ \cite[Section III.2 Proposition 15]{bou}.  Let $\aaa$ be an ideal of $R$. 
The filtration $(\aaa^0, \aaa^1, \aaa^2, \ldots)$ of $R$ is called the {\it $\aaa$-adic filtration} of $R$, and we say that $R$ is {\it complete with respect to $\aaa$} if $R$ is complete with respect to the $\aaa$-adic filtration.  We say that a filtration $\ff$ of $R$ is an {\it $\aaa$-filtration} of $R$ if each $\ff_i$ contains some power of $\aaa$.
If $\ff$ is an $\aaa$-filtration of $R$, then $\widehat{\ff}$ is an $\widehat{\aaa}$-filtration of $\RRR_\ff$ and there is a unique continuous $R$-homomorphism $\RRR_\aaa \longrightarrow \RRR_\ff$.

Let us say that $f \in R[X]$ is {\it $\aaa$-Weierstrass (in $R[X]$)} if $f$ is monic and $f = X^{\deg f}$ in $(R/\aaa)[X]$, that is, if $f_{\deg f} = 1$ and $f_i \in \aaa$ for all $i < \deg f$.  We say that $f \in R[[X]]$ is {\it $\aaa$-distinguished of order $n$ (in $R[[X]]$)}, where $n$ is a nonnegative integer,  if $f_n$ is a unit modulo $\aaa$ and $f_i \in \aaa$ for all $i < n$, or equivalently, $f$ is associate to $X^n$ in $(R/\aaa)[[X]]$. 
A polynomial of degree $n$ is $\aaa$-Weierstrass if and only if it is monic and $\aaa$-distinguished of order $n$. 

\begin{theorem}[Weierstrass preparation theorem]\label{Weierstrass} Let $R$ be a ring and $\aaa$ an ideal of $R$ such that $R$ is complete with respect to some $\aaa$-filtration of $R$.  For any $f \in R[[X]]$ that is  $\aaa$-distinguished of order $n$, where $n$ is a nonnegative integer, there exists a unique $\aaa$-Weierstrass polynomial $P \in R[X]$ and a unique unit $U \in R[[X]]$ such that $f = UP$; equivalently, $P$ is the unique monic polynomial in $R[X]$ of least degree that is divisible by $f$ in $R[[X]]$; moreover, one has $\deg P = n$.
\end{theorem}

A well-known and elegant proof of Theorem \ref{Weierstrass} in the case where $R$ is complete with respect to $\aaa$ and $\aaa$ is the unique maximal ideal of $R$ is given in \cite[Theorems IV.9.1--2]{lang}. The theorem is also proved in \cite{oma} in the case where $R$ is complete with respect to $\aaa = (f_0, f_1, \ldots, f_{n-1})$.  However, the hypothesis that $R$ is complete with respect to $\aaa$ imposes unnecessary restrictions on applicability of the theorem, since for an arbitrary ideal $\aaa$ of a ring $R$ the $\aaa$-adic completion $\RRR_\aaa$ of $R$ is complete with respect to the $\widehat{\aaa}$-filtration $(\widehat{\aaa^0}, \widehat{\aaa^1},  \widehat{\aaa^2}, \ldots)$ but may not be complete with respect to $\widehat{\aaa}$  if $\aaa$ is not finitely generated \cite[Exercise III.2.12 and Section III.2 Proposition 16  Corollary 2]{bou}.  Thus our generalization of the Weierstrass preparation theorem to arbitrary complete filtered rings allows for applications to the completion of any ring with respect to any of its ideals.

In Section \ref{sec:3} we adapt the proof of \cite[Theorems IV.9.1--2]{lang} to yield a concise proof of Theorem \ref{Weierstrass}.   We call the polynomial $P$ of Theorem \ref{Weierstrass} the {\it Weierstrass polynomial (in $R[X]$) associated to $f$} and denote it by $P_f = P_{f,R}$, and we denote $U$ by $U_f = U_{f,R}$.  Algorithm \ref{walg} of Section \ref{sec:5} is an algorithm for computing $P_f$ and $U_f$ to any desired degree of accuracy, correct to within $\AAA^N$ for any desired positive integer $N$, where $\AAA = \aaa R[[X]] + XR[[X]]$, under the assumption that one knows $f$ to within $\AAA^{(n+1)N}$ and one has an algorithm for performing the ring operations in $R/\aaa^i$ for $i \leq (n+1)N$.

It should be noted that any ``algorithm'' using the ring operations of a ring $R$ requires infinite precision if some elements of $R$ are not finitely specifiable, which is the case, for example, if $R$ is the ring of integers $\mathcal{O}_K$ of some local field $K$. Nevertheless, there are honest algorithms to which, for example, one can input a sufficient approximation to $f \in \mathcal{O}_K[X]$ modulo a power of $\pi \mathcal{O}_K[X]$, where $\pi$ is a uniformizer of $K$, to obtain the irreducible factors of $f$ to a given desired degree of accuracy.  See \cite{can} \cite{chi} \cite{guar} \cite{guar1} \cite{guar2}, for example.   Although some of the algorithms in this paper, such as Algorithms \ref{irredtestalg}, \ref{irredtest2a}, \ref{mainfactoralg}, and \ref{mainfactortheorem}, are equipped with ``oracles'' that tacitly perform the ring operations in $\RRR_\aaa[X]$ or $\RRR_\aaa[[X]]$ for some ideal $\aaa$ of a ring $R$ to infinite precision, we also provide honest algorithms by showing that such algorithms are stable modulo powers of the ideals $\widehat{\aaa}[X]$ and $\widehat{\aaa}[[X]]+X\RRR_\aaa[[X]]$, respectively.  In particular, this is done in Algorithms \ref{walg}, \ref{computeweierstrassalg}, \ref{preplemmaalg},  \ref{mainfactortheoremeffective}, and \ref{irredtestalg2}.

In Section \ref{sec:3}  we prove the following irreducibility criterion for elements of the UFD $R[[X]]$ for any PID $R$.

\begin{theorem}\label{irredtest}
Let $R$ be a PID and $f \in R[[X]]$. Then $f$ is irreducible in $R[[X]]$ if and only if one of the following conditions holds.
\begin{enumerate}
\item $f_0 = 0$ and $f_1$ is a unit in $R$.
\item $f_0$ is associate to a power of some prime $\pi \in R$ and, writing $f = \pi^k f'$, where $\pi \nmid f' \in R[[X]]$, either of the following conditions holds.
\begin{enumerate}
\item $k = 1$ and $f'_0$ is a unit in $R$.
\item $k = 0$ and $P_{f, \RRR_{(\pi)}}$ is irreducible in $\RRR_{(\pi)}[X]$.
\end{enumerate}
\end{enumerate}
Moreover, if $f_0$ is associate to a power of a prime $\pi \in R$ and $\pi \nmid f$, then $R[[X]]/(f) \cong \RRR_{(\pi)}[X]/(P)$ as $R[X]$-algebras, where $P = P_{f, \RRR_{(\pi)}}$ is the Weierstrass polynomial in  $\RRR_{(\pi)}[X]$ associated to $f$.  
\end{theorem}

The following corollary is immediate.

\begin{corollary}\label{irredtestalg0}
Let $R$ be a PID and $f \in R[X]$, and suppose that $f$ is $(\pi)$-Weierstrass and $f_0$ is associate to a power of $\pi$ for some prime $\pi \in R$.  Then $R[[X]]/(f) \cong \RRR_{(\pi)}[X]/(f)$ as $R[X]$-algebras, and therefore $f$ is irreducible in $R[[X]]$ if and only if $f$ is irreducible in $\RRR_{(\pi)}[X]$.
\end{corollary}

The rest of this paper is organized as follows.  Section \ref{sec:2} generalizes \cite[Propositions 3.3 and 3.4]{gil1}.  Section \ref{sec:3} provides several primality criteria for certain distinguished formal power series over an arbitrary ring (Theorems \ref{maintheorem1} and \ref{maintheorem1b}), yielding proofs of Theorems \ref{irredtest2}, \ref{Weierstrass}, and \ref{irredtest} above.  Section \ref{sec:4} contains a proof of Theorem \ref{quotientrings}, and Section \ref{sec:5} provides algorithms for computing associated Weierstrass polynomials.  Section \ref{sec:6}, the main result of which is Algorithm \ref{mainfactortheoremeffective}, provides algorithms for factoring formal power series over a PID; further work would likely yield algorithms that are more efficient than these.  Section \ref{sec:7} provides irreducibility criteria for certain distinguished formal power series over an integral domain (Theorem \ref{maintheorem2}), yielding an alternative (and shorter) proof of Theorem \ref{irredtest}. Finally, Section \ref{sec:8} provides an irreducibility test for formal power series over a factorial number ring that have no nonunit constant or square divisors.

\section{Elementary observations on factoring formal power series}\label{sec:2}

The results in this section generalize \cite[Propositions 3.3 and 3.4]{gil1}.

\begin{proposition}\label{simplefactor1}
Let $R$ be an integral domain and $f \in R[[X]]$.  If $f_0$ is irreducible in $R$, then $f$ is irreducible in $R[[X]]$.
\end{proposition}

\begin{proof}
If $f = gh$, then $f_0 = g_0 h_0$, whence either $g_0$ or $h_0$ is a unit in $R$, and therefore either $g$ or $h$ is a unit in $R[[X]]$.
\end{proof}

\begin{proposition}\label{simplefactor}
Let $R$ be a ring and $f \in R[[X]]$.  Suppose that $f_0 = ab$, where $(a, b) = R$.  Choose $r, s \in R$ with $r a + s b = 1$.  Set $g_0 = 0$ and let
$$g_n = f_n - rs \sum_{i = 1}^{n-1} g_i g_{n-i}$$ for all positive integers $n$.  Then $g = \sum_{n = 0}^\infty g_n X^n \in R[[X]]$ is the unique solution to the equation
$$f = (a + s g)(b + r g)$$
with $g_0 = 0$, and one has $(a+sg, b+rg) = R[[X]]$.
\end{proposition}

\begin{proof}
The equation for $g$ is equivalent to $g = f -f_0 -rsg^2$, which, assuming $g_0 = 0$, is equivalent to the given recurrence relation.  Also $r(a+sg) + s(b+rg) = 1+2rsg$ is a unit in $R[[X]]$.
\end{proof}

\begin{corollary}\label{PIDlemma}
Let $R$ be an integral domain and $f \in R[[X]]$.
\begin{enumerate}
\item If $f$ is irreducible in $R[[X]]$, then $(f_0) \neq R$ and for any $a, b \in R$ with $f_0 = ab$ and $(a,b) = R$ one has $(a) = R$ or $(b) = R$.  Moreover, the converse holds if $(f_1) = R$.
\item If $R$ is a PID and $f$ is irreducible in $R[[X]]$, then either $(f) = (X)$ or $f_0$ is associate to a power of a prime in $R$.
\item Suppose that $R$ is a UFD but not a PID.  Then there exist irreducible power series in $R[[X]]$ whose constant term is neither zero nor associate to a power of a prime.  In fact, if  $\pi, \sigma \in R$ are nonassociate primes with $(\pi,\sigma) \neq R$, then $\pi \sigma + XU$ is irreducible in $R[[X]]$ for any unit $U \in R[[X]]$.
\end{enumerate}
\end{corollary}

Two elements $a$ and $b$ of a ring $R$ are said to be {\it coprime} if $(a,b) = R$.  In a PID this condition holds if and only if $a$ and $b$ share no nonunit factors.

\begin{proposition}\label{factorbasic}
Let $R$ be a ring and $f \in R[[X]]$.  Suppose that $f_0 = a_1 a_2 \cdots a_k$, where the $a_i$ are pairwise coprime in $R$.   Choose $b_1, b_2, \ldots, b_k \in R$ such that 
$\sum_{i = 1}^k b_i \prod_{j \neq i} a_j = 1.$
Then there exists a unique $g \in XR[[X]]$ such that
$$f = (a_1 + b_1 g)(a_2 + b_2 g)\cdots (a_k + b_k g).$$
Moreover, the $a_i + b_i g$ are pairwise coprime in $R[[X]]$.   In particular, if $f_0$ is a product of $k$ pairwise coprime nonzero nonunits of $R$, then $f$ is a product of $k$ pairwise coprime nonzero nonunits of $R[[X]]$.
\end{proposition}

\begin{proof}
The given equation for $g$ has the form 
$$g = f - f_0 + c_2 g^2 + c_3g^3 + \cdots + c_n g^n,$$
where $c_i \in R$ for all $i$, which, assuming $g_0 = 0$, is equivalent to a recurrence relation for the $g_i$
of the form $$g_i = f_i +  F_i(g_1, g_2, \ldots, g_{i-1}),$$
where $F_i \in R[X_1, X_2, \ldots, X_{i-1}]$ for all $i$.	The existence and uniqueness of a solution $g \in XR[[X]]$ follows, and the rest of the proposition is then clear.
\end{proof}

\begin{corollary}
Let $R$ be a PID and $f \in R[[X]]$.  If $f_0$ has at least $k$ pairwise nonassociate prime factors in $R$, then the prime factorization of $f$ in the UFD $R[[X]]$ is of length at least $k$.
\end{corollary}

\section{Primality criteria}\label{sec:3}

In this section we prove Theorems \ref{Weierstrass}, \ref{irredtest}, and \ref{irredtest2} of the introduction. 

For any ring $R$ and any nonnegative integer $n$, let $\tau_n$ and $\alpha_n$ denote the $R$-linear operators on $R[[X]]$ acting by $\tau_n: f \longmapsto \sum_{i = 0}^\infty f_{n+i} X^{i}$ and $\alpha_n: f \longmapsto f - X^n\tau_n(f) = \sum_{i = 0}^{n-1} f_i X^i$, respectively.  

\begin{proof}[Proof of Theorem \ref{Weierstrass}]
Let $\ff$ be any $\aaa$-filtration of $R$ with respect to which $R$ is complete.  Let $V \in R[[X]]$.  Then $Vf \equiv V_0 f_nX^n \, (\Mod \, (\aaa [[X]]+X^{n+1}R[[X]]))$, and therefore $Vf$ is automatically an $\aaa$-Weierstrass polynomial of degree $n$ provided it is a monic polynomial of degree at most $n$, in which case $V$ is a unit in $R[[X]]$.  So it remains to show that there is a unique $V \in R[[X]]$ such that $Vf$ is a monic polynomial of degree $n$.  Now, $Vf$ is a monic polynomial of degree $n$ if and only if $\tau(Vf) = 1$, where $\tau = \tau_n$.  Since  $\tau(Vf) = \tau(V (\alpha(f) + X^n\tau(f))) = \tau(V \alpha(f)) + V\tau(f)$, where $\alpha = \alpha_n$, this in turn is equivalent to
$$\tau(V \alpha(f)) + V\tau(f) = 1.$$
  Put $Z = V\tau(f)$.  Since $\tau(f)$ is a unit, the above equation is equivalent to
$$(T+I)(Z) = \tau(\alpha(f) \tau(f)^{-1}Z) + Z = 1,$$
where $I$ is the identity operator and $T$ the operator $\tau \circ (\alpha(f)\tau(f)^{-1} -)$ on $R[[X]]$.
Since $\alpha(f) \in \aaa R[X]$, the operator $T$ maps $\aaa^i [[X]]$ into $\aaa^{i+1} [[X]]$, so $\im T^i \subseteq\aaa^i [[X]]$, for all $i$.  Thus for every $j$ one has $\im T^i \subseteq \aaa^i[[X]] \subseteq \ff_j[[X]]$ for sufficiently large $i$. Therefore the operator $T+I$ on $R[[X]]$ is invertible with inverse $(T+I)^{-1} = \sum_{i = 0}^\infty (-1)^i T^i$.  Thus the equation $(T+I)(Z) = 1$ is equivalent to
$Z = (T+I)^{-1}(1)$, or equivalently $V = \tau(f)^{-1} (T+I)^{-1}(1)$.  Such a $V \in R[[X]]$ therefore exists and is unique.
\end{proof}

Next, Theorem \ref{irredtest} follows immediately from Corollary \ref{PIDlemma}(2) and the following theorem.

\begin{theorem}\label{maintheorem1}
Let $R$ be a ring and $\aaa$ an ideal of $R$, let $f \in R[[X]]$ be $\widehat{\aaa}$-distinguished in $\RRR_\aaa[[X]]$, and let $P$ be the Weierstrass polynomial in $\RRR_\aaa[X]$ associated to $f$.  Suppose that the $R[[X]]$-homomorphism $R[[X]]/(f) \longrightarrow \RRR_{(f_0)}[[X]]/(f)$ is an isomorphism, which holds if $f_0$ a nonzerodivisor in $R$ or $R$ is Noetherian.  Then each of the following conditions implies the next.
\begin{enumerate}
\item $\aaa^k \subseteq (f_0)$ for some positive integer $k$.
\item The $R$-homomorphism $\RRR_{(f_0)} \longrightarrow \RRR_\aaa$ is an isomorphism. 
\item 
There is a (unique) $R[X]$-isomorphism $R[[X]]/(f) \longrightarrow {\RRR_\aaa[X]/(P)}$ so that the diagram
\begin{eqnarray*}
\SelectTips{cm}{11}\xymatrix{
R[[X]]/(f) \ar[dr] \ar[r] \ar[d] & {\RRR_\aaa[X]/(P)} \ar[d]  \\
{\RRR_{(f_0)}[[X]]/(f)} \ar[r] & {\RRR_\aaa[[X]]/(f)}
}
\end{eqnarray*}
of $R[X]$-homomorphisms (in fact, $R[X]$-isomorphisms) is commutative.
\item $R[[X]]/(f) \cong \RRR_{\aaa}[X]/(P)$ as rings.
\item $f$ is prime in $R[[X]]$ if and only if $P$ is prime in $\RRR_{\aaa}[X]$.
\end{enumerate}
\end{theorem}

To prove Theorem \ref{maintheorem1} we use the following four lemmas.

\begin{lemma}\label{preplemma2a}
Let $R$ be ring and $\aaa = (a_1, \ldots, a_k)$ a finitely generated ideal of $R$.  There is a surjective $R$-homomorphism $\varphi: R[[X_1, \ldots, X_k]] \longrightarrow \RRR_\aaa$ acting by $f \longmapsto f(a_1, \ldots, a_k)$ with  $\ker \varphi = \bigcap_n (\nn+\aaa^nR[[X_1, \ldots, X_k]])$, where $\nn = (a_1-X_1, \ldots, a_k-X_k)$, and one has $\widehat{\aaa^n} = \aaa^n \RRR_\aaa$ 
 for every nonnegative integer $n$.  Moreover, if $R$ is Noetherian then $\ker \varphi = \nn$.
\end{lemma}

\begin{proof}
This follows from the proof of  \cite[Theorems 17.4 and 17.5]{nag}.
\end{proof}

\begin{lemma}\label{preplemma2}
Let $R$ be ring and $a \in R$.  Suppose  that $a$ is a nonzerodivisor in $R$ or $R$ is Noetherian.  Then $\RRR_{(a)} \cong R[[X]]/(a-X)$ as $R$-algebras and $\widehat{(a^n)} = a^n\RRR_{(a)}$ for every nonnegative integer $n$.  Moreover, if $a$ is a nonzerodivisor in $R$, then $a$ is a nonzerodivisor in $\RRR_{(a)}$.
\end{lemma}

\begin{proof}
By Lemma \ref{preplemma2a} we may suppose that $a$ is a nonzerodivisor in $R$.  Let $\varphi: R[[X]] \longrightarrow \RRR_{(a)}$ denote the surjective $R$-homomorphism $f \longmapsto f(a)$.  Let $f \in \ker \varphi$.  By Lemma \ref{preplemma2a} one has $f = (a-X)g_n + X^nh_n$ for some $g_n, h_n \in R[[X]]$ for all $n \geq 1$.  Then $f_0 = a(g_n)_0$ and $f_i = a(g_n)_i-(g_n)_{i-1}$ for all $i < n$.  Since $a$ is a nonzerodivisor, this implies that $(g_n)_i$ is uniquely determined for $i < n$, so $(g_m)_i = (g_n)_i$ if $i < n \leq m$.  Thus $G = \lim_n g_n$ exists in $R[[X]]$ and $f = \lim_n ( (a-X)g_n+ X^n h_n) = (a-X)G$.  Therefore $\ker \varphi = (a-X)$.

Suppose now that $af = 0$ in $R[[X]]/(a-X)$, where $f \in R[[X]]$; say, $af= (a-X)g$ with $g \in R[[X]]$.  Then $a f_i = ag_i-g_{i-1}$ for all $i \geq 1$, so $a$ divides $g$ in $R[[X]]$.  Say $g = ah$ with $h \in R[[X]]$.  Then $af = a(a-X)h$, so $f = (a-X)h$ and therefore $f = 0$ in $R[[X]]/(a-X)$.  Thus $a$ is a nonzerodivisor in $\RRR_{(a)} \cong R[[X]]/(a-X)$.
\end{proof}

\begin{lemma}\label{newlemma}
Let $R$ be a ring and $f \in R[[X]]$.  Then $S = R[[X]]/(f)$ is complete in the $f_0S$-adic topology.  Moreover, if $f_0$ is a nonzerodivisor in $R$ or $R$ is Noetherian, then $S$ is complete with respect to $f_0 S$ and the $R[[X]]$-homomorphism $R[[X]]/(f) \longrightarrow \RRR_{(f_0)}[[X]]/(f)$ is an isomorphism.
\end{lemma}

\begin{proof}
Let $a = f_0$, let $\pi : R[[X]] \longrightarrow S$ be the quotient map,  and let $x = \pi(X)$. We first show that $S$ is complete in the $aS$-adic topology.  Let $g_i \in S$ for all nonnegative integers $i$ and $s_n = \sum_{i = 0}^{n-1} g_i a^i$ for all positive integers $n$.  We claim that the sequence $\{s_n\}$ converges $aS$-adically in $S$.  Choose $h_i \in R[[X]]$ with $\pi(h_i) = g_i$.  Now $f =0$ in $S$ implies $a = \pi(g)$, where $g = a-f \in XR[[X]]$.  Let $t = \sum_{i=0}^\infty h_i g^i \in R[[X]]$.  Then for any positive integer $n$ we have $\pi(t) = s_n + a^n\pi\left(\sum_{i = n}^\infty h_i g^{i-n}\right)$, so $\pi(t) \equiv s_n \, (\Mod \, a^n S)$ for all $n$.  Therefore $S$ is complete in the $aS$-adic topology.

Suppose now that $R$ is Noetherian.  Then $R[[X]]$ is Noetherian and complete with respect to $XR[[X]]$, so by \cite[Theorem 8.7]{mat} $S$ is complete with respect to $XR[[X]]$ as an $R[[X]]$-module, whence $S$ is complete with respect to $xS$.  Therefore, since $S$ is Noetherian, it follows from \cite[Exercise 8.2]{mat} that $S$ is also complete with respect to $aS \subseteq xS$.  

Suppose, alternatively, that $f_0$ is a nonzerodivisor in $R$.  We show that $S$ is complete with respect to $a S$.  To prove this we need only show that $S$ is Hausdorff in the $aS$-adic topology, or equivalently $\bigcap_n a^n S = (0)$, that is, $\bigcap_n (f,a^n) = (f)$ in $R[[X]]$. Now, $(f,a^n) = (f, g^n)$, where $g = a-f \in XR[[X]]$.  Let $h \in \bigcap_n (f,a^n)$.  For every nonnegative integer $n$ we may write 
$h = f k_n + g^n l_n$ with $k_n, l_n \in R[[X]]$.  Since $f_0$ is a nonzerodivisor in $R$, by induction on $i$ this equation uniquely determines $(k_n)_i$ for $i <n$.   Therefore $K = \lim_n k_n$ exists in $R[[X]]$, so $h = \lim_n (f k_n+ g^n l_n) = fK$.  Therefore $\bigcap_n (f, a^n) = (f)$. 

Supposing, then, that $R$ is Noetherian or $a = f_0$ is a  nonzerodivisor in $R$, we have $\widehat{S}_{aS} \cong S$.  Moreover, by Lemma \ref{preplemma2} we have  $\RRR_{(a)}/a^n \RRR_{(a)} \cong R/(a^n)$, so $\RRR_{(a)}[[X]]/(f)$ maps homomorphically onto the inverse system $\frac{(R/(a^n))[[X]]}{(f)}$ of $R$-algebras.  Therefore $\RRR_{(a)}[[X]]/(f)$ maps homomorphically into $\widehat{S}_{aS} \cong \varprojlim \frac{(R/(a^n))[[X]]}{(f)}$, so there is a commutative diagram
\begin{eqnarray*}
\SelectTips{cm}{11}\xymatrix{
{S}  \ar[dr]^\cong \ar[r] & {\RRR_{(a)}[[X]]/(f)} \ar[d]  \\
 & {\widehat{S}_{aS}}
}
\end{eqnarray*}
of $R[[X]]$-homomorphisms.  Thus $\varphi: S \longrightarrow {\RRR_{(a)}[[X]]/(f)}$ is an inclusion.  Finally, for any $b_i \in R$ one has $\varphi(\sum_i b_i (a-f)^i) = \sum_i b_i a^i$, so $\RRR_{(a)} \subseteq \im \varphi$ and therefore $\varphi$ is surjective, hence an isomorphism.
\end{proof}

\begin{lemma}\label{preplemma4}
Let $R$ be a ring, $\aaa$ an ideal of $R$, and $P$ an $\aaa$-Weierstrass polynomial in $R[X]$.  We have the following.
\begin{enumerate}
\item Suppose that $\bigcap_n \aaa^n = (0)$.  If $F = hP$ with $F \in R[X]$ and $h \in R[[X]]$, then $h \in R[X]$.
\item If $R$ is complete with respect to an $\aaa$-filtration, then the $R[X]$-homomorphism $R[X]/(P) \longrightarrow R[[X]]/(P)$ is an isomorphism. 
\end{enumerate}
\end{lemma}

\begin{proof} 
To prove (1), let $n = \deg P$, and let $m$ be any positive integer such that $n+m > \deg F$.  For all $i \geq m$ we have
$$0 = F_{n+i} = h_i+ (h_{i+1}P_{n-1}+ h_{i+2}P_{n-2}  \cdots + h_{n+i}P_0),$$
whence $h_i \in \aaa$.  But then the same equation and an obvious inductive argument imply that $h_i \in \aaa^j$ for all $i \geq m$ and all $j$, whence
$h_i = 0$ for all $i \geq m$. This proves (1).

Next, suppose that $R$ is complete with respect to an $\aaa$-filtration, let $\varphi$ denote the $R[X]$-homomorphism $R[X] \longrightarrow R[[X]]/(P)$, and let $x = \varphi(X)$.  By (1) one has $\ker \varphi = (P)$.  Let $\sum_i a_i X^i \in R[[X]]$.  For any nonnegative integer $i$, writing $i = qn + r$ with $q,r\in \ZZ$ and $0 \leq r < n$, we have $$X^i = X^{qn}X^r \equiv (-P_0 -P_1X - \cdots -P_{n-1}X^{n-1})^q X^r \ (\Mod PR[[X]]).$$  It follows that  $x^i = b^q x^r$ for some $b \in \aaa$.  Since $b \in \aaa$, the series $\sum_q a_{qn+r} b^q$ converges in $R$ for all $r$.  Therefore
$$\sum_i a_i x^i = \sum_{r = 0}^{n-1} \left(\sum_q a_{qn+r} b^q \right) x^r \in \im \varphi.$$
Thus $\varphi$ is surjective.  This proves (2).
\end{proof}

\begin{proof}[Proof of Theorem \ref{maintheorem1}]
Since $(f_0) \subseteq R\cap \widehat{\aaa} = \aaa$, it follows that (1) implies (2).
By Lemmas \ref{newlemma} and \ref{preplemma4}(2), (2) implies (3), and the remaining implications are clear.
\end{proof}

The following result generalizes Theorem \ref{maintheorem1}.

\begin{theorem}\label{maintheorem1b}
Let $R$ be a ring and $\aaa$ an ideal of $R$, let $f \in R[[X]]$ be $\widehat{\aaa}$-distinguished in $\RRR_\aaa[[X]]$, and let $P$ be the Weierstrass polynomial in $\RRR_\aaa[X]$ associated to $f$.  Suppose that the $R[[X]]$-homomorphism $R[[X]]/(f) \longrightarrow \RRR_{(f_0)}[[X]]/(f)$ is an isomorphism (which holds if $f_0$ is a nonzerodivisor in $R$ or $R$ is Noetherian).  Then the following conditions are equivalent.
\begin{enumerate}
\item There is a (unique) $(X)$-adically continuous $R[X]$-isomorphism $R[[X]]/(f) \longrightarrow {\RRR_\aaa[X]/(P)}$.
\item There is a (unique) $R[X]$-isomorphism $R[[X]]/(f) \longrightarrow {\RRR_\aaa[X]/(P)}$ so that the diagram
\begin{eqnarray*}
\SelectTips{cm}{11}\xymatrix{
R[[X]]/(f) \ar[dr] \ar[r] \ar[d] & {\RRR_\aaa[X]/(P)} \ar[d]  \\
{\RRR_{(f_0)}[[X]]/(f)} \ar[r] & {\RRR_\aaa[[X]]/(f)}
}
\end{eqnarray*}
of $R[X]$-homomorphisms (in fact, $R[X]$-isomorphisms) is commutative.
\item The $R[[X]]$-homomorphism $R[[X]]/(f) \longrightarrow \RRR_\aaa[[X]]/(f)$ is an isomorphism.
\item The $R[[X]]$-homomorphism $\RRR_{(f_0)}[[X]]/(f) \longrightarrow \RRR_\aaa[[X]]/(f)$ is an isomorphism.
\end{enumerate}
Moreover, if $\bigcap_n \aaa^n = (0)$ and $f_0$ a nonzerodivisor in  $\RRR_\aaa$, then the above conditions are equivalent to the following.
\begin{enumerate}
\item[(5)] The $R$-homomorphism $\RRR_{(f_0)} \longrightarrow \RRR_\aaa$ is surjective and $R \cap f_0 \RRR_{\aaa} = f_0R$.
\item[(6)] The $R$-homomorphism $R/f_0 R \longrightarrow \RRR_\aaa/f_0 \RRR_\aaa $ is an isomorphism.
\end{enumerate}
\end{theorem}

To prove the theorem we need also the following two lemmas.

\begin{lemma}\label{nzdlemma}
Let $R \subseteq S$ be rings and let $f \in R[[X]]$ with $a = f_0$ a nonzerodivisor in $S$.  Then $R  \cap a S= a R$ if and only if $R[1/a] \cap S = R$ in  $S[1/a]$, if and only if $b = a c$ for $b \in R$ and $c \in S$ implies $c \in R$.  Moreover, if those conditions hold, then $f$ is a nonzerodivisor in $S[[X]]$, and one has the following.
\begin{enumerate}
\item $R[[X]]  \cap f S[[X]]= f R[[X]]$.
\item $R[[X]][1/f] \cap S[[X]] = R[[X]]$ in $S[[X]][1/f]$.
\item $g = f^n h$ for $g \in R[[X]]$, $h \in S[[X]]$, and $n \in \ZZ_{>0}$ implies $h \in R[[X]]$.
\end{enumerate}
\end{lemma}

\begin{proof}
Clearly the three conditions on $a$ are equivalent and imply that $f$ is a nonzerodivisor in $S[[X]]$. Thus it suffices to prove statement (3) with $n = 1$.  Suppose $g = fh$ with $g \in R[[X]]$ and $h \in S[[X]]$.  One has $g_0 = f_0 h_0 \in R$, so $h_0 \in R$.  If $h_j \in R$ for all $j < i$, t  hen $f_0 h_i = g_i - f_1h_{i-1} - \cdots - f_i h_0 \in  R$, so $h_i \in R$.  Therefore $h_i \in R$ for all $i$ by induction on $i$, whence $h \in R[[X]]$.  
\end{proof}

\begin{lemma}\label{preplemma3}
Let $R$ be a ring and $\aaa$ an ideal of $R$, and let $f \in R[[X]]$.  Suppose that the $R[[X]]$-homomorphism $R[[X]]/(f) \longrightarrow \RRR_{(f_0)}[[X]]/(f)$ is an isomorphism.  Then the following conditions are equivalent.
\begin{enumerate}
\item The $R[[X]]$-homomorphisms in the commutative diagram
\begin{eqnarray*}
\SelectTips{cm}{11}\xymatrix{
R[[X]]/(f) \ar[dr]  \ar[d] & \\
{\RRR_{(f_0)}[[X]]/(f)} \ar[r] & {\RRR_\aaa[[X]]/(f)}
}
\end{eqnarray*}
are isomorphisms.
\item The $R[[X]]$-homomorphism $R[[X]]/(f) \longrightarrow \RRR_\aaa[[X]]/(f)$ is an isomorphism.
 \item The $R[[X]]$-homomorphism $\RRR_{(f_0)}[[X]]/(f) \longrightarrow \RRR_\aaa[[X]]/(f)$ is an isomorphism.
\end{enumerate}
Moreover, if $\bigcap_n \aaa^n = (0)$ and $f_0$ is a nonzerodivisor in $\RRR_\aaa$, then the above conditions are equivalent to the following.
\begin{enumerate}
\item[(4)] The $R$-homomorphism $\RRR_{(f_0)} \longrightarrow \RRR_\aaa$ is surjective and $R \cap f_0 \RRR_{\aaa} = f_0R$.
\item[(5)] The $R$-homomorphism $R/f_0 R \longrightarrow \RRR_\aaa/f_0 \RRR_\aaa$ is an isomorphism.
\end{enumerate}
\end{lemma}

\begin{proof}
If $f_0$ is a nonzerodivisor in $R$ or $R$ is Noetherian, then by Lemma \ref{newlemma} the $R[[X]]$-homomorphism  $R[[X]]/(f) \longrightarrow \RRR_{(f_0)}[[X]]/(f)$ is an isomorphism.   Statements (1) through (3) are clearly equivalent.  Suppose that $\bigcap_n \aaa^n = (0)$ and $a = f_0$ is a nonzerodivisor in $\RRR_\aaa$.  If (2) holds, then 
one has $R$-isomorphisms
$$\frac{R}{aR} \longrightarrow 
\frac{R[[X]]}{(f,X,a)} \longrightarrow \frac{\RRR_{\aaa}[[X]]}{(f,X,a)} \longrightarrow \frac{\RRR_{\aaa}}{a\RRR_\aaa},$$
so (5) holds.  

Suppose that (5) holds.   Let $b \in \RRR_\aaa$. Since the $R$-homomorphism  $R \longrightarrow \RRR_\aaa/a \RRR_\aaa $ is surjective, we may recursively find $b_i \in R$ and $c_i \in \RRR_\aaa$ so that $b = c_n a^n + \sum_{i = 0}^{n-1} b_i a^i$ for all positive integers $n$.   It follows that $b = \sum_{i = 0}^\infty b_i a^i$.   Therefore (4) holds.

Suppose that (4) holds.  Let $\psi$ denote the homomorphism in statement (3), which is clearly surjective.  Let $\varphi$ denote the homomorphism $R[[X]] \longrightarrow \RRR_\aaa[[X]]/(f)$.  Let $g \in \ker \varphi$, so $g = fh$ for some $h \in \RRR_\aaa[[X]]$.  Since $R \cap f_0 \RRR_\aaa = f_0 R$ and $f_0$ is a nonzerodivisor in $\RRR_\aaa$, by Lemma \ref{nzdlemma} one has $h \in R[[X]]$.  Therefore $\ker \varphi = (f)$.  It then follows from Lemma \ref{newlemma} and the commutative diagram of statement (1) that $\psi$ is injective and therefore an isomorphism.  Therefore (3) holds.
\end{proof}

\begin{proof}[Proof of Theorem \ref{maintheorem1b}]
Because the image of $R[X]$ in $R[[X]]/(f)$ is $(X)$-adically dense, statements (1) and (2) are equivalent by Lemma \ref{newlemma}.  Statements (2) and (3) are equivalent by  Lemma \ref{preplemma4}(2), and the rest of the theorem follows from Theorem \ref{maintheorem1} and Lemma \ref{preplemma3}. 
 \end{proof}

An alternative proof of Theorem \ref{irredtest} is provided in Section \ref{sec:7}.

Next, let $R$ be a UFD.  For any $f \in R[[X]]$ we let the {\it content $c(f)$ of $f$} be any gcd in $R$ of the coefficients of $f$, which is unique up to associate.  
Let $a \in R$.  If $a \neq 0$, write $\omega(a) = \omega(a,R)$ for the total number of nonassociate primes in any prime factorization of $a$ in $R$, and write $\omega(0) = \omega(0, R) = \infty$.   Note that $\omega(a) = 0$ if and only if $a$ is a unit in $R$, and $\omega(a) = 1$ if and only if $a$ is associate to $\pi^k$ for some prime $\pi \in R$ and some positive integer $k$.  Theorem \ref{irredtest} implies the following irreducibility test for elements of $R[[X]]$.

\begin{algorithm}\label{irredtestalg}
Let $R$ be a PID and $f$ a nonzero nonunit in $R[[X]]$.   Choose a prime $\pi \in R$ dividing $f_0$.  Let $k$ denote the largest nonnegative integer such that  $\pi^k \mid c(f)$.  Let $l$ denote the number of irreducible polynomials, counting multiplicities, in any irreducible factorization of  $P_{f/c(f), \RRR_{(\pi)}}$ in $\RRR_{(\pi)}[X]$. Given $\omega(f_0), \omega(f_1), k,l \in  \ZZ_{\geq 0} \cup \{\infty\}$, the following algorithm determines whether or not $f$ is irreducible in $R[[X]]$.
\begin{enumerate} 
\item If $\omega(f_0) = \infty$, then return IRREDUCIBLE if $\omega(f_1) = 0$ and REDUCIBLE otherwise.
\item If $\omega(f_0) > 1$, then return REDUCIBLE.
\item If $k+l= 1$, then return IRREDUCIBLE. Otherwise return REDUCIBLE.
\end{enumerate}
\end{algorithm}

Next we prove Theorem \ref{irredtest2} of the introduction.

\begin{lemma}\label{preplemmab}
Let $\aaa$ be an ideal of a ring $R$.
\begin{enumerate}
\item Let $f, g_i \in R[X]$ be monic polynomials with $f = g_1 \cdots g_k$.  Then $f$ is $\aaa$-Weierstrass if $g_i$ is $\aaa$-Weierstrass for each $i$.
\item Let $f, g_i \in R[[X]]$ with $f = g_1 \cdots g_k$, and let $n$ be a nonnegative integer.  Then $f$ is $\aaa$-distinguished of order $n$ if each $g_i$ is $\aaa$-distinguished of order $n_i$ for some nonnegative integer $n_i$ and $n = n_1 + \cdots + n_k$.
\end{enumerate}
Moreover, the converses of statements (1) and (2) both hold if $\aaa$ is prime.
\end{lemma}

\begin{proof}
It suffices to prove (2) and its converse for $\aaa$ prime.  We may assume without loss of generality that $k = 2$.  Say $f = gh$.  If $g$ and $h$ are $\aaa$-distinguished  of order $a$ and $b$, respectively, then $f$ is associate to $X^a X^b = X^{a+b}$ in $(R/\aaa)[[X]]$, whence $f$ is $\aaa$-distinguished of order $n = a+b$.  Conversely, suppose that $f$ is $\ppp$-distinguished of order $n$, where $\ppp = \aaa$ is prime.  Then $UX^n = f = gh$ in $(R/\ppp)[[X]]$ for some unit $U \in (R/\ppp)[[X]]$.  Since $X$ is prime in the domain $(R/\ppp)[[X]]$, it follows that $g = VX^a$ and $h = UV^{-1}X^b$ in $(R/\ppp)[[X]]$, where $a$ and $b$ are nonnegative integers with $n = a+b$ and $V$ is a unit in $(R/\ppp)[[X]]$.  Therefore $g$ and $h$ are $\ppp$-distinguished of order $a$ and $b$, respectively.
\end{proof}

\begin{proof}[Proof of Theorem \ref{irredtest2}]
By Corollary \ref{PIDlemma}(2) we may assume $f_0 \neq 0$ and $f_0$ is associate to a power of some prime $\pi \in R$.  Write $f = \pi^k f'$, where $\pi \nmid f' \in R[X]$.   Let $P = P_{f', \RRR_{(\pi)}}$ and $U = U_{f', \RRR_{(\pi)}}$.  By Lemma \ref{preplemma4}(1) we have $U \in \RRR_{(\pi)}[X]$, so $f =\pi^k PU$ is a factorization of $f$ in the UFD $\RRR_{(\pi)}[X]$.  Moreover, by Lemma \ref{preplemmab} any monic factor of $P$ is $\pi\RRR_{(\pi)}$-Weierstrass.  Therefore one has $f = \pi^k PU = gh$ for some $g,h \in \RRR_{(\pi)}[X]$ with $g_0, h_0 \in \pi \RRR_{(\pi)}$ if and only if  one of the following three conditions holds: (1) $k \geq 2$; (2) $k = 1$ and $P \neq 1$; or (3) $k = 0$ and $P$ is reducible in $\RRR_{(\pi)}[X]$.  Finally, by Theorem \ref{irredtest}, one of these three conditions holds if and only if $f$ is reducible in $R[[X]]$.
\end{proof}

Theorem \ref{irredtest2} may be expressed alternatively as follows.

\begin{algorithm}\label{irredtest2a}
Let $R$ be a PID and $f$ a nonzero polynomial in $R[X]$ with $f_0$ a  nonunit in $R$.   Choose a prime $\pi \in R$ dividing $f_0$.  Let $\chi \in \{0,1\}$ be equal to $0$ if $f = gh$ for $g,h \in \RRR_{(\pi)}[X]$ with $g_0, h_0 \in \pi\RRR_{(\pi)}$, and $1$ otherwise.   Given $\omega(f_0), \omega(f_1), \chi \in \ZZ_{\geq 0} \cup \{\infty\}$, the following algorithm determines whether or not $f$ is irreducible in $R[[X]]$.
\begin{enumerate} 
\item If $\omega(f_0) = \infty$, then return IRREDUCIBLE if $\omega(f_1) = 0$ and REDUCIBLE otherwise.
\item If $\omega(f_0) > 1$, then return REDUCIBLE.
\item If $\chi = 0$ then return REDUCIBLE.  Otherwise return IRREDUCIBLE.
\end{enumerate}
\end{algorithm}

\section{Integral domains arising as quotient rings}\label{sec:4}

In this section we prove Theorem \ref{quotientrings} of the introduction from the following three lemmas.  The first generalizes \cite[Lemma 3.1.4]{mcd}.

\begin{lemma}\label{preplemma}
Let $R \subseteq S$ be rings and $f \in S[[X]]$.  Suppose that $f_0$ is a nonzerodivisor in $S$ and the $R$-algebra homomorphism $R \longrightarrow S/f_0S$ is surjective, and let $d \in R \cap f_0S$.
\begin{enumerate}
\item There exist $g \in R[[X]]$ with $g_0 = d$ and $h \in S[[X]]$ such that $g = hf$.  Moreover, if $R \cap f_0S \subseteq dS$ then any such $h$ is a unit in $S[[X]]$.
\item The $g_n$ and $h_n$ may be defined recursively as follows.  Let $\Pi  \subseteq R$ be a system of representatives for $S/f_0 S$, and for any $s \in S/f_0 S$ let $s \, \Mod f_0$ denote the unique element of $\Pi$ with image $s$ in $S/f_0S$.  Set $g_0 = d \in R$ and $h_0 = f_0^{-1}d \in S \subseteq S[1/f_0]$, and for any positive integer $n$ define $g_n$ and $h_n$ recursively as follows:
\begin{enumerate}
\item $g_n = \left( \sum_{i=0}^{n-1} h_if_{n-i} \right) \, \Mod f_0 \in R$.
\item $h_n =  f_0^{-1}\left(g_n -  \sum_{i=0}^{n-1} h_if_{n-i} \right) \in S \subseteq S[1/f_0]$.
\end{enumerate}
\item Let $f' \in S[[X]]$ with $f'_0 S = f_0 S$, and let $g', h'$ be defined as $g,h$ are defined in (2) so that $g' = h'f'$.  If $f' \equiv f \ (\Mod \, (f_0, X)^N)$ for some positive integer $N$, then $g' \equiv g \ (\Mod \, X^N R[[X]])$ and $h' \equiv h \ (\Mod \, (f_0, X)^{N-1}).$
\end{enumerate}
\end{lemma}


\begin{proof}
The series $g \in R[[X]]$ and $h\in S[[X]]$ as defined in (2) are clearly well-defined and satisfy the equation $g = hf$.  If $R \cap f_0S \subseteq  dS$, then $f_0 R \subseteq R \cap f_0 S \subseteq dS \subseteq f_0 S$, so $f_0 S = dS$, and therefore any $h$ as in statement (1) is a unit in $S[[X]]$.  An easy induction shows that, if $f_i' \equiv f_i \ (\Mod \, f_0^{N-i}S)$ for all $i \leq N$, then $g_i' = g_i$ and $h_i' \equiv h_i  \ (\Mod \, f_0^{N-1-i}S)$ for all $i \leq N-1$.  This proves (3).
\end{proof}

Statement (3) of the above lemma implies that the result is effective in the sense that $g$ and $h$ can be computed to any desired degree of accuracy 
by computing $f$ to a corresponding specified degree of accuracy.  See Algorithm \ref{preplemmaalg} for an algorithm based on the lemma.

Our second lemma is an alternative version of Theorem \ref{quotientrings}.

\begin{lemma}\label{quotientrings1}
Let $R$ be a PID.  The integral domains arising as quotient rings of $R[[X]]$ are, up to isomorphism, precisely the following: $R[[X]]$, $R$, $(R/\mm)[[X]]$, $R/\mm$, and $\RRR_{\mm}[\alpha]$, where $\mm$ is any maximal ideal of $R$ and $\alpha$ is any element of the algebraic closure of the quotient field of $\RRR_\mm$ that is a root of some $\widehat{\mm}$-Weierstrass polynomial in $\RRR_\mm[X]$.
\end{lemma}

\begin{proof}
The maximal ideals of $R[[X]]$ are of height two and of the form $(\pi, X)$, where $\pi \in R$ is prime \cite[ Section 1.1 Example 1]{mat}, and one has $R[[X]]/(\pi,X) \cong R/(\pi)$.  Since $R[[X]]$ is a UFD, every height one prime ideal $\ppp$ of $R[[X]]$ is principal \cite[Theorem 20.1]{mat}, equal to $(f)$ for some irreducible element $f$ of $R[[X]]$ and contained in $(\pi, X)$ for some prime $\pi \in R$.
If $f_0 = 0$ then $\ppp = (X)$ and $R[[X]]/\ppp \cong R$.  Otherwise by Corollary \ref{PIDlemma}(2) $f_0$ is equal to a unit times a power of $\pi$.  If $\pi \mid f$ then $\ppp = (\pi)$ and $R[[X]]/\ppp \cong (R/(\pi))[[X]]$.  Otherwise $f \in R[[X]]\backslash \pi R[[X]]$ and therefore by Theorem \ref{irredtest} one has $R[[X]]/\ppp \cong \RRR_\mm[\alpha]$, where  $\mm = (\pi)$ and $\alpha$ is a root in the algebraic closure $K$ of the quotient field of $\RRR_\mm$ of the Weierstrass polynomial in $\RRR_\mm[X]$ associated to $f$.

Conversely, suppose that $\alpha \in K$ is a root of some irreducible $\widehat{\mm}$-Weierstrass polynomial $P$ in $\RRR_\mm[X]$.  If $P_0 = 0$, then $\alpha = 0$ and $\RRR_\mm[\alpha] = \RRR_\mm \cong R[[X]]/(\pi-X)$.  If $P_0 \neq 0$, then by Lemma \ref{preplemma} there exist $f \in R[[X]]$ and a unit $U \in \RRR_\mm[[X]]$ such that $f = UP$ and $f_0$ is equal to a power of $\pi$, and therefore since $P = P_f$ one has  $\RRR_\mm[\alpha] \cong  \RRR_\mm[X]/(P) \cong R[[X]]/(f)$ by Theorem \ref{irredtest}.  Thus in either case $\RRR_\mm[\alpha]$ is isomorphic to a quotient ring of $R[[X]]$.
\end{proof}

Our third lemma generalizes \cite[Lemma 3.1.5]{mcd}.

\begin{lemma}\label{what}
Let $R$ be an integral domain with quotient field $K$ and with absolute integral closure $R^+$, let $\alpha \in R^+$, let $\aaa$ be a proper ideal of $R$, and let $n$ be a positive integer.  The following conditions are equivalent.
\begin{enumerate}
\item $\alpha$ is a root of some $\aaa$-Weierstrass polynomial in $R[X]$ of degree $n$.
\item $\alpha^n \in \aaa R[\alpha]$.
\end{enumerate}
Suppose, furthermore, that $\aaa$ is prime.  Then the above conditions are equivalent to the following.
\begin{enumerate}
\item[(3)] The minimal polynomial of $\alpha$ over $K$ is an $\aaa$-Weierstrass polynomial in $R[X]$ of degree $d \leq n$.
\end{enumerate}
Finally, suppose that $R$ is a Henselian local ring with maximal ideal $\aaa = \mm$.  Then $R^+$ is a local ring with maximal ideal $\mm^+$ lying over $\mm$, and the above conditions are equivalent to the following.
\begin{enumerate}
\item[(4)]  The minimal polynomial of $\alpha$ over $K$ has degree $d \leq n$ and its constant term is in $\mm$.
\item[(5)]  $\alpha \in \mm^+$ and $\deg(\alpha, K) \leq n$.
\end{enumerate}
\end{lemma}

\begin{proof}
Let $f$ denote the minimal polynomial $\alpha$ over $K$ and $d$ its degree.   Suppose that $\alpha$ is a root of an $\aaa$-Weierstrass polynomial $P \in R[X]$ of degree $n$.  Then $\alpha^n = -(P_0+\cdots+P_{n-1}\alpha^{n-1}) \in \aaa R[\alpha]$.  Conversely, suppose that $\alpha^n \in \aaa R[\alpha]$.   Since $f$ is monic of degree $d$ every element of $R[\alpha]$ may be written in the form $a_0 + \cdots + a_{n-1}\alpha^{d-1}$ with each $a_i \in R$, and since $\alpha^n \in \aaa R[\alpha]$ we may write $\alpha^n$ in the same form with each $a_i \in \aaa$.  Thus $\alpha$ is a root of  the polynomial $F = a_0 + a_1 X + \cdots + a_{d-1} X^{d-1} - X^n$ in $R[X]$.  Since $f$ has degree $d$, if $n < d$ then $F = 0$ and therefore $a_n - 1 = 0$, contradicting $\aaa \neq R$.  It follows that $n \geq d$ and $F$ is $\aaa$-Weierstrass.  This proves (1) and (2) equivalent.  Next, under these same hypotheses, if $\aaa$ is prime then $f$ is $\aaa$-Weierstrass by Lemma \ref{preplemmab} since $f$ divides $F$ and $F$ is $\aaa$-Weierstrass.  Thus (1) implies (3).  Conversely, if (3) holds, then, multiplying $f$ by any $\aaa$-Weierstrass polynomial of degree $n-d$, we see that $\alpha$ is a root of an $\aaa$-Weierstrass polynomial of degree $n$.  Thus (1), (2), and (3) are equivalent if $\aaa$ is prime.

Now  suppose that $R$ is a Henselian local ring with maximal ideal $\aaa = \mm$.    It is proved in \cite{art} that $R^+$ is a local ring with maximal ideal $\mm^+$ lying over $\mm$.  If condition (3) (and therefore (2)) holds, then $\deg(\alpha,K) = d \leq n$ and $\alpha \in \sqrt{\mm R[\alpha]} \subseteq \sqrt{\mm^+} = \mm^+$, so (5) holds.    Next, supposing that condition (5) holds, one has $f_0 = -\alpha(f_1 + f_2 \alpha + \cdots + \alpha^{d-1}) \in \mm^+ \cap R = \mm$.  Therefore (5) implies (4).  Finally, we show that (4) implies (3).  Suppose that (4) holds, and suppose to obtain a contradiction that $f$ is not $\mm$-Weierstrass.  Let $k$ be the smallest nonnegative integer such that $f_k \notin \mm$, so $1 \leq k \leq d-1$.  Then $f = X^k G$ in $(R/\mm)[X]$ for some $G \in (R/\mm)[X]$ with $G_0 \neq 0$.  Since $(X^k, G) = (R/\mm)[X]$ and $R$ is Henselian it follows that $f = hg$ for monic polynomials $h,g \in R[X]$ with $h \equiv X^k \ (\Mod \, \mm R[X])$ and $g \equiv G \ (\Mod \, \mm R[X])$, contradicting the irreducibility of $f$ in $R[X]$.
\end{proof}

Finally, we obtain Theorem \ref{quotientrings}.

\begin{proof}[Proof of Theorem \ref{quotientrings}]
The theorem follows immediately from Lemmas \ref{quotientrings1} and \ref{what}.
\end{proof}

\begin{remark}
The ring $\RRR_\mm[\alpha]$ for any $\alpha$ as in Theorem \ref{quotientrings} is a complete local one-dimensional Noetherian domain with maximal ideal $(\pi, \alpha) = \sqrt{(\pi)}$ and residue field $R/\mm$, where $\pi$ is a generator of $\mm$.  Moreover, $\RRR_\mm[\alpha]$ is a DVR if and only if it is integrally closed, if and only if its maximal ideal $(\pi, \alpha)$ is principal, if and only if (by \cite[Proposition 18]{serre}) $\RRR_\mm[\alpha] = \RRR_\mm[\beta]$ for some  $\beta \in \RRR_\mm^+$ whose minimal polynomial in $\RRR_\mm[X]$ is Eisenstein.  Thus, the quotient rings of $R[[X]]$ that are DVRs, besides possibly $R[[X]]/(X) \cong R$, are precisely those integral domains isomorphic for some maximal ideal $\mm$ of $R$ to $(R/\mm)[[X]]$ or to $\mathcal{O}_K$ for some totally ramified finite extension $K$ of the quotient field of $\RRR_\mm$.
\end{remark}




\section{Computing Weierstrass polynomials}\label{sec:5}

Let $R$ be a ring and $\aaa$ an ideal of $R$, let $$\AAA = \aaa R[[X]] + X R[[X]] = \{f \in R[[X]]: f_0 \in \aaa\},$$ and let $k$ be a positive integer.  For all $i \leq k$ let $\Pi_i \subseteq R$ be a system of representatives for $R/\aaa^i$.  For all $r \in R$ and all $i \leq k$ we let $r \, \Mod \, \aaa^i$ denote the unique element of $\Pi_i$ congruent to $r$ modulo $\aaa^i$.  Moreover, for all $f \in R[[X]]$ we let $f \, \Mod \, \AAA^k$ denote the unique polynomial in $R[X]$ of degree at most $k-1$ congruent to $f$ modulo $\AAA^k$ with $i$th coefficient in $\Pi_{k-i}$ for all $i < k$.  One has
$$f \, \Mod \, \AAA^k = \sum_{i = 0}^{k-1} (f_i \, \Mod \, \aaa^{k-i})X^i,$$
and therefore the function $\Mod \, \AAA^k: R[[X]] \longrightarrow R[X]$ is completely determined by the set of functions $\Mod \, \aaa^i: R \longrightarrow \Pi_i$ for $i \leq k$, and vice versa.

The following proposition and algorithm, which are inspired by \cite[Proposition 3]{sum} and its proof, show that $P_f$ and $U_f$ are stable under approximations of $f$ and can be approximated to any desired degree of accuracy provided that one can perform the ring operations in $R/\aaa^i$ for sufficiently large $i$.

\begin{proposition}\label{estimate} Let $R$ be a ring and $\aaa$ an ideal of $R$ such that $R$ is complete with respect to some $\aaa$-filtration of $R$.  Let $\AAA = \aaa R[[X]] +XR[[X]]$, and let $n$ be a nonnegative integer and $N$ a positive integer.  Let $f, g \in R[[X]]$ be $\aaa$-distinguished of order $n$ with $f \equiv g \ (\Mod \, \AAA^{(n+1)N})$. Then $$P_f \equiv P_g \ (\Mod \, \AAA^{N+1})$$ and $$U_f \equiv U_g \ (\Mod \, \AAA^N).$$ Moreover, for all $i \leq (n+1)N$ let $\Pi_i \subseteq R$ be a system of representatives for $R/\aaa^i$.  Assume $g = f \, \Mod \, \AAA^{(n+1)N}$.  Let $t = \tau_n(g)^{-1} \, \Mod \AAA^{(n+1)N-n}$.  Define $S_{i} \in R[[X]]$ recursively, with $S_{0} = 1$ and $$S_{i} = \tau_n (t\alpha_n(g) S_{i-1}) \, \Mod \, \AAA^{(n+1)N-n(i+1)}$$
for all positive integers $i \leq N-1$.
Then one has
$$U_f^{-1} \equiv  t\sum_{i = 0}^{N-1} (-1)^i S_{i} \ (\Mod \, \AAA^N)$$
and
$$P_f \equiv  gt \sum_{i = 0}^{N-1} (-1)^i S_{i} \ (\Mod \, \AAA^N).$$
\end{proposition}

\begin{proof}
We use the notation as in the proof of Theorem \ref{Weierstrass}. Let $V_f = U_f^{-1}$ and $T_f = \tau \circ (\alpha(f)\tau(f)^{-1}-)$, and define $V_g$ and $T_g$ similarly.   One has
\begin{equation}\label{congruence}
V_f =  \tau(f)^{-1}\sum_{i = 0}^\infty (-1)^i T_f^i(1) \equiv \tau(f)^{-1}\sum_{i = 0}^{N-1} (-1)^i T_f^i(1) \ (\Mod \, \aaa ^N R[[X]])
\end{equation}
and a similar congruence for $V_g$.
Now, $\alpha(\AAA^i) \subseteq \AAA^i$ for all $i$ and $\tau(\AAA^i) = \AAA^{i-n}$ for all $i \geq n$.  Thus $\alpha(f) \equiv \alpha(g) \ (\Mod \, \AAA^{(n+1)N})$
and $\tau(f) \equiv \tau(g) \ (\Mod \, \AAA^{(n+1)N-n})$,
whence
$$T_f(1) \equiv T_g(1) \ (\Mod \, \AAA^{(n+1)N-2n}).$$
By induction on $i$ one has
$$T_f^i(1) \equiv T_g^i(1) \ (\Mod \, \AAA^{(n+1)N-n(i+1)})$$
for all $i \leq N- 1$.  Therefore, since $(n+1)N-nN = N$, by the congruence (\ref{congruence}) above one has $V_f \equiv V_g \ (\Mod \, \AAA^{N})$
and thus $U_f \equiv U_g \ (\Mod \, \AAA^{N}).$
If $f$ is a unit then $P_f = 1 = P_g$; otherwise $f \in \AAA$ and one has
$$P_f - P_g = (V_f-V_g)f+V_g(f-g) \in \AAA^{N+1}.$$
Therefore $P_f \equiv P_g \ (\Mod \, \AAA^{N+1})$.  Finally, by induction on $i$ one has $S_i \equiv T_f^i(1) \ (\Mod \,\AAA^{(n+1)N-n(i+1)})$ for all $i \leq N-1$, and the congruences for $U_f^{-1}$ and $P_f$ stated in the proposition follow from the congruence (\ref{congruence}) above.
\end{proof}

\begin{algorithm}\label{walg}
Let $R$ be a ring and $\aaa$ an ideal of $R$ such that $R$ is complete with respect to some $\aaa$-filtration of $R$.  Let $\AAA = \aaa R[[X]] +XR[[X]]$, let $n$ be a nonnegative integer and $N$ a positive integer, and let $f \in R[[X]]$ be $\aaa$-distinguished of order $n$.   For all $i \leq (n+1)N$ let $\Pi_i \subseteq R$ be a system of representatives for $R/\aaa^i$, and let $g = f \, \Mod \, \AAA^{(n+1)N}$.  Given the data $(R,n,N, g)$ and $(\Mod \, \aaa^i)_{i \leq (n+1)N}$, the following algorithm returns $U_f^{-1} \, \Mod \, \AAA^N$ and $P_f \, \Mod \, \AAA^N$.
\begin{enumerate}
\item Compute $t = \tau_n(g)^{-1} \, \Mod \, \AAA^{(n+1)N-n}$.
\item Set $S_0 = 1$.  Compute
$$S_i = \tau_n(t\alpha_n(g) S_{i-1}) \, \Mod \, \AAA^{(n+1)N-n(i+1)}$$
for all positive integers $i \leq N-1$.
\item Compute
$$V =   \left(t\sum_{i = 0}^{N-1} (-1)^i S_{i} \right) \, \Mod \, \AAA^N$$
and
$$P = (gV) \, \Mod \, \AAA^N.$$
\item Return $U_f^{-1} \, \Mod \, \AAA^N =  V$ and $P_f \, \Mod \, \AAA^N = P$.
\end{enumerate}
\end{algorithm}

\begin{proof}
This follows immediately from Proposition \ref{estimate}.
\end{proof}

Next, in the case where $R$ is a UFD and $\aaa$ is prime, we provide an alternative algorithm for computing associated Weierstrass polynomials to any desired degree of accuracy, given any algorithm for factoring polynomials in $R[X]$ to any desired degree of accuracy.  The algorithm is based on a suggestion by Chase Franks for computing Weierstrass polynomials over $R = \mathcal{O}_K$, where $K$ is a finite extension of $\QQ_p$.

Let $R$ be a ring and $a \in R$.  We say that a {\it prime factorization of $a$ in $R$} is a sequence $\Phi = (\pi_0,\pi_1, \pi_2, \ldots, \pi_k)$ such that $a = \pi_0 \pi_1 \pi_2 \cdots \pi_k$, where $\pi_0 \in R$ is a unit and $\pi_i \in R$ is prime 
for $i > 0$.  The nonnegative integer $k$ is called the {\it length} of the prime factorization $\Phi$, which we denote by $l(\Phi)$. 
Write $\Phi_i = \pi_i$ for all $i \leq l(\Phi)$.  If $a$ is a nonzerodivisor, then all prime factorizations of $a$ in $R$ are unique in the obvious sense.  

\begin{proposition}\label{prep}
Let $R$ be a UFD and $\ppp$ a prime ideal of $R$ such that $R$ is complete with respect to some $\ppp$-filtration of $R$.  Let $f \in R[X]$ be $\ppp$-distinguished, and suppose that $f_0 \neq 0$ and $\Phi$ is a prime factorization of $f$ in $R[X]$, where the $\Phi_i$ are indexed so that $(\Phi_i)_0$ for $i > 0$ is a unit in $R$ if and only if $i > k$, where $k \in \ZZ_{\geq 0}$.  Let $u$ be the leading coefficient of the polynomial $\Phi_1 \cdots \Phi_k$.  Then $(u^{-1}, \Phi_1, \ldots, \Phi_k)$ is a prime factorization of $P_f$ in $R[X]$.  Also, one has $U_f \in R[X]$, and $(u \Phi_0, \Phi_{k+1}, \Phi_{k+2}, \ldots, \Phi_{l(\Phi)})$ is a prime factorization of $U_f$ in $R[X]$.
\end{proposition}

\begin{proof}
Let $P = P_f$ and $U = U_f$.  One has $U \in R[X]$ by Lemma \ref{preplemma4}(1), so $PU = f =  \Phi_0(\Phi_1 \cdots \Phi_k)(\Phi_{k+1} \cdots \Phi_{l(\Phi)})$ in $R[X]$, with $\Phi_i$ irreducible in $R[X]$ for all $i > 0$.  For all $i > k$, clearly $\Phi_i$ is not a unit in $R$ times a $\ppp$-Weierstrass polynomial, and therefore by Lemma \ref{preplemmab} one has $\Phi_i \nmid P$ in $R[X]$.  Also, for all $0< i \leq k$, since $(\Phi_i)_0 \nmid U_0$, one has $\Phi_i \nmid U$.  It follows that $\Phi_1 \cdots \Phi_k$ divides $P$ and $\Phi_{k+1}\cdots \Phi_{l(\Phi)}$ divides $U$ in $R[X]$; let the quotients in $R[X]$ be $P'$ and $U'$, respectively, so $P'U' = \Phi_0$ is a unit in $R[X]$.  Thus $P'$ and $P'^{-1}\Phi_0$ are units in $R[X]$, hence units in $R$, whence $P' = u^{-1}$ and $U' = u \Phi_0$.
\end{proof}

Let $R$ be a ring and $\aaa$ an ideal of $R$, and let $a \in R$.  Let us say that a {\it mod $\aaa$ prime factorization of $a$ in $R$} is a list $\Psi = (a_0, a_1, \ldots, a_k)$ of elements of $R$, with $a_0$ a unit in $R$, such that $a$ has a prime factorization $(\pi_0, \pi_1, \ldots, \pi_k)$ in $R$ with $\pi_i \equiv a_i \ (\Mod \aaa)$ for all $i$.
Write $\Psi_i = a_i$ for all $i \leq k$.  

\begin{algorithm}\label{computeweierstrassalg}
Let $R$ be a UFD and $\ppp$ a prime ideal of $R$ such that $R$ is complete with respect to some $\ppp$-filtration of $R$.  Let $\PPP = \ppp R[X] +XR[X]$, and let $N$ be an positive integer.   Let $f \in R[[X]]$ be $\ppp$-distinguished of order $n$ with $f_0 \neq 0$.  For all $i \leq (n+1)N$ let $\Pi_i \subseteq R$ be a system of representatives for $R/\ppp^i$.  Suppose that $\Psi$ is a mod $\PPP^{(n+1)N}$ prime factorization of $g = f \, \Mod \, \PPP^{(n+1)N}R[[X]]$ in $R[X]$, where the $\Psi_i$ are indexed so that $(\Psi_i)_0$ for $i > 0$ is a unit in $R$ if and only if $i > k$.  Given the data $(R,\Psi,k)$, the following algorithm returns a mod $\PPP^N$ prime factorization of $P_g$ in $R[X]$ with $P_g \equiv P_f \, (\Mod \PPP^N)$.
\begin{enumerate}
\item Let $u$ be the leading coefficient of the polynomial $\Psi_1 \cdots \Psi_k$.
\item Return $(u^{-1}, \Psi_1, \ldots, \Psi_k)$. 
\end{enumerate}
\end{algorithm}

\begin{proof}
This follows from Propositions \ref{estimate} and \ref{prep}.
\end{proof}

\section{Factoring algorithms}\label{sec:6}

In this section we provide an algorithm for computing any number of coefficients of the irreducible factors, counting multiplicities, of a polynomial in $R[[X]]$ for a PID $R$.  

The following algorithm computes a prime factorization of any power series in $R[[X]]$, provided that one can factor corresponding Weierstrass polynomials over the relevant completions of $R$ to infinite precision, where $R$ is a PID.

\begin{algorithm}\label{mainfactoralg}
Let $R$ be a PID with quotient field $K$ and $f$ a nonzero formal power series in $R[[X]]$. Let $\Phi(c(f))$ be a prime factorization of $c(f)$ in $R$.  Let $r$ be the least nonnegative integer such that $f_r \neq 0$.   Let $\PP$ be a complete set of nonassociate prime factors of $f_r/c(f)$ in $R$.    For each $\pi \in \PP$, let $(1, \Phi_{\pi, 1},\cdots, \Phi_{\pi, l_\pi})$ be a prime factorization of $P_{f/c(f)X^r, \RRR_{(\pi)}}$ in $\RRR_{(\pi)}[X]$, and for each $i \leq l_\pi$ let $\phi_{\pi, i}$ be an element of  $R[[X]]$ such that $\phi_{\pi, i} = V_{\pi, i} \Phi_{\pi, i}$ for some unit $V_{\pi, i} \in \RRR_{(\pi)}[[X]]$ and $(\phi_{\pi, i})_0$ is a power of $\pi$, defined as in Lemma \ref{preplemma}(2).    Let $V =  (f/c(f) X^r)\left(\prod_{\pi \in \PP}\prod_{i \leq l_\pi} \phi_{\pi, i}\right)^{-1} \in K[[X]]$.  Given the data $(f,r, \Phi(c(f)), V)$ and $(\phi_{\pi, i})_{\pi \in \PP, i \leq l_\pi}$, the following algorithm returns a prime factorization of $f$ in $R[[X]]$.
\begin{enumerate}
\item Write $\Phi(c(f)) = (u, \tau_1, \ldots, \tau_s)$ and $(\phi_{\pi, i})_{\pi \in \PP, i \leq l_\pi}= (g_1, \ldots, g_t)$.
\item Return $(uV, \tau_1, \ldots, \tau_s, X, X, \ldots, X, g_1, \ldots, g_t)$.
\end{enumerate}
\end{algorithm}

\begin{proof}
Replacing $f$ with $f/c(f)X^r$, we may assume without loss of generality that $r = 0$ and $c(f) = 1$.  Factor $f_0$ as a unit $u$ times a product $\prod_{\pi \in \PP} \pi^{n_\pi}$.  Replacing $f$ with $u^{-1}f$, we may assume that $u = 1$.  Choose $b_\pi \in R$ with $\sum_{\pi \in \PP} b_\pi \prod_{\sigma \in  \PP \backslash\{\pi\}} \sigma^{n_\sigma} = 1$.  By Proposition \ref{factorbasic} there is a unique $g \in XR[[X]]$ such that  $f =  \prod_{\pi \in \PP} (\pi^{n_\pi} + b_\pi g)$.  Let $f_\pi = \pi^{n_\pi} + b_\pi g$ and $f_\pi' = f/f_\pi = \prod_{\sigma \in  \PP \backslash\{\pi\}} (\sigma^{n_\sigma} + b_\sigma g)$. 
In $\RRR_{(\pi)}[X]$ one has  $P_{f_\pi} = P_f = \Phi_{\pi, 1} \cdots \Phi_{\pi, l_\pi}$ and $f_\pi U_f = f U_{f_\pi}$.  By Lemma \ref{preplemmab} each $\Phi_{\pi, i}$ is a unit $u_{\pi, i}$ in $\RRR_{(\pi)}$ times a $\pi \RRR_{(\pi)}$-Weierstrass polynomial. 
Since $P_{\phi_{\pi, i}} = u_{\pi, i}^{-1} \Phi_{\pi, i}$ and $(\phi_{\pi, i})_0$ is a power of $\pi$, by Theorem \ref{irredtest} each $\phi_{\pi, i}$ is irreducible in $R[[X]]$.  Moreover, one has $$f_\pi = U_{f_\pi} P_{f_\pi} = f_\pi'^{-1}U_f \Phi_{\pi, 1} \cdots \Phi_{\pi, l_\pi} = V_\pi  \phi_{\pi, 1} \cdots \phi_{\pi, l_\pi},$$ with  each $\phi_{\pi, i}$ irreducible in $R[[X]]$, where $V_\pi =  (f_\pi' V_{\pi, 1} \cdots V_{\pi, l_\pi})^{-1} U_f$ is a unit in $\RRR_{(\pi)}[[X]]$; by Lemma \ref{nzdlemma} one has $V_\pi \in R[[X]]$; and since $(V_\pi)_0$ divides $\pi^{n_\pi}$ in $R$ and $(V_\pi)_0$ is a unit in $\RRR_{(\pi)}$, it follows that $(V_\pi)_0$ is a unit in $R$ and therefore $V_\pi$ is a unit in $R[[X]]$.  Thus $f = \prod_{\pi \in \PP} f_\pi = (\prod_{\pi \in \PP} V_\pi)(  \prod_{\pi \in \PP}\prod_{ i \leq l_\pi} \phi_{\pi, i})$ and $V = \prod_{\pi \in \PP} V_\pi$ is a unit in $R[[X]]$, whence $(V, g_1, \ldots, g_t)$ is a prime factorization of $f$ in $R[[X]]$.
\end{proof}

For polynomials the algorithm above simplifies as follows.

\begin{algorithm}\label{mainfactortheorem}
Let $R$ be a PID with quotient field $K$ and $f$ a nonzero polynomial in $R[X]$.   Let $\Phi(c(f))$ be a prime factorization of $c(f)$ in $R$.  Let $r$ be the least nonnegative integer such that $f_r \neq 0$.  Let $\PP$ be a complete set of nonassociate prime factors of $f_r/c(f)$ in $R$.   For each $\pi \in \PP$, let $(1, \Phi_{\pi, 1},\cdots, \Phi_{\pi, l_\pi})$ be a prime factorization of $f/c(f)X^r$ in $\RRR_{(\pi)}[X]$, where the $\Phi_{\pi, i}$ are indexed so that $(\Phi_{\pi, i})_0$ a unit in $\RRR_{(\pi)}$ if and only if $i > k_\pi$, where $k_\pi  \in \ZZ_{\geq 0}$.  For each $i \leq k_\pi$ let $\phi_{\pi, i}$ be an element of $R[[X]]$ such that $\phi_{\pi, i} = V_{\pi, i} \Phi_{\pi, i}$ for some unit $V_{\pi, i} \in \RRR_{(\pi)}[[X]]$ and $(\phi_{\pi, i})_0$ is a power of $\pi$, defined as in Lemma \ref{preplemma}(2).   Let $V =  (f/c(f) X^r)\left(\prod_{\pi \in \PP}\prod_{i \leq k_\pi} \phi_{\pi, i}\right)^{-1} \in K[[X]]$.  Given the data $(f,r, \Phi(c(f)), V)$ and $(\phi_{\pi, i})_{\pi \in \PP, i \leq k_\pi}$, the following algorithm returns a prime factorization of $f$ in $R[[X]]$.
\begin{enumerate}
\item Write $\Phi(c(f)) = (u, \tau_1, \ldots, \tau_s)$ and $(\phi_{\pi, i})_{\pi \in \PP, i \leq k_\pi} = (g_1, \ldots, g_t)$.
\item Return $(uV, \tau_1, \ldots, \tau_s, X, X, \ldots, X, g_1, \ldots, g_t)$.
\end{enumerate}
\end{algorithm}

\begin{proof}
The proof is similar to that of Algorithm \ref{mainfactoralg}.
\end{proof}

The following algorithm follows as in the proof of Lemma \ref{preplemma}.

\begin{algorithm}\label{preplemmaalg}
Let $R \subseteq S$ be rings and $f \in S[[X]]$.  Suppose that $f_0$ is a nonzerodivisor in $S$ and the $R$-algebra homomorphism $R \longrightarrow S/f_0S$ is surjective, and let $d \in R \cap f_0S$.  Let $N > 1$ be an integer.  For all $i \leq N-1$ let $\Pi_i \subseteq S$ be a system of representatives of $S/f_0^i S$, with $\Pi_1 \subseteq R$,  and for any $s \in S/f_0^i S$ let $s \, \Mod f_0^i$ denote the unique element of $\Pi_i$ with image $s$ in $S/f_0^iS$. Let $R' = R[\Pi] \subseteq S$, where $\displaystyle \Pi = \bigcup_{i \leq N-1} \Pi_i$.  Let $f' = f \, \Mod \, (f_0, X)^N \in R'[X]$.   Given the data $(R'[f_0^{-1}],d,N,f')$ and $(\Mod \, f_0^i)_{i \leq N-1}$, the following algorithm returns $g \, \Mod \, X^N$ and $h \, \Mod \, (f_0,X)^{N-1} \in R'[X]$ for some $g \in R[[X]]$ and $h \in S[[X]]$  with $g_0 = d$ and $g = hf$. 
\begin{enumerate}
\item Set $g_0' = d$ and $h_0' = f_0^{-1}d$.
\item For any positive integer $n \leq N-1$ define $g_n'$ and $h_n'$ recursively as follows: $$g_n' = \left( \sum_{i=0}^{n-1} h_i' f_{n-i}'  \right) \, \Mod \, f_0,$$
$$h_n' =  f_0^{-1}\left(g_n' -  \sum_{i=0}^{n-1}  h_i' f_{n-i}' \right)  \, \Mod \, f_0^{N-1-n}.$$
\item Return
$$g \, \Mod \, X^N  = \sum_{i = 0}^{N-1} g_i' X^i$$
and
$$h \, \Mod \, (f_0,X)^{N-1} = \sum_{i = 0}^{N-2} h_i' X^i.$$
\end{enumerate}
\end{algorithm}

Finally, in the following algorithm, which is our main result on factoring formal power series, we indicate how to implement Algorithm \ref{mainfactortheorem} to any desired degree of accuracy without requiring infinite precision in $R[[X]]$ or in the completions of $R$.

\begin{algorithm}\label{mainfactortheoremeffective}
Let $R$ be a PID, let $f$ be a nonzero polynomial in $R[X]$, and let $N >1$ an integer.    Let $\Phi(c(f))$ be a prime factorization of $c(f)$ in $R$.  Let $r$ be the least nonnegative integer such that $f_r \neq 0$.  Let $f' = f \, \Mod \, X^{N+r} R[[X]] \in R[X]$.  Let $\PP$ be a complete set of nonassociate prime factors of $f_r/c(f)$ in $R$.  For all $\pi \in \PP$, let $(1, \Psi_{\pi, 1},\ldots, \Psi_{\pi, l_\pi})$ be a mod $(f_r/c(f),X)^N$ prime factorization of $f/c(f)X^r$ in $\RRR_{(\pi)}[X]$.   Let $R' = R[(\Psi_{\pi, i})_0^{-1}: \pi \in \PP, i \leq l_\pi\}]$.  For each $i \leq l_\pi$, let $t_{\pi, i} \in \ZZ_{\geq 0}$ with $(\Psi_{\pi, i})_0 \RRR_{(\pi)} = \pi^{t_{\pi, i}} \RRR_{(\pi)}$.  Let $v = (f_r/c(f)) \prod_{\pi \in \PP}\prod_{i \leq l_\pi} \pi^{-t_{\pi,i}} \in R$.  For each $j \leq m_\pi = (N-1)\max\{t_{\pi, i}: i \leq l_\pi\}$, let $\Pi_{\pi, j} \subseteq R$ be a system of representatives for $R/\pi^jR$, and for any $s \in \RRR_{(\pi)}/\pi^j\RRR_{(\pi)}$ let $s \, \Mod \, \pi^j$ denote the unique element of $\Pi_{\pi, j}$ with image $s$ in $\RRR_{(\pi)}/\pi^j\RRR_{(\pi)}$. Given the data $(R', N,r,f',\Phi(c(f)),v)$, $(\Psi_{\pi, i}, t_{\pi, i})_{\pi \in \PP, i \leq l_\pi}$, and $(\Mod \, \pi^j)_{\pi \in \PP, j \leq m_\pi}$, the following algorithm returns a mod $X^{N} R[[X]]$ prime factorization of  $f$ in $R[[X]]$.
\begin{enumerate}
\item For each $\pi \in \PP$ reindex the $\Psi_{\pi, i}$ so that $t_{\pi, i} = 0$ if and only if $i > k_\pi$, where $k_\pi  \in \ZZ_{\geq 0}$.
\item Using Algorithm \ref{preplemmaalg}, for all $i \leq k_\pi$  compute $\psi_{\pi, i}' = \psi_{\pi, i} \, \Mod X^NR[[X]]$, where $\psi_{\pi, i } = W_{\pi, i} \Psi_{\pi, i} \in R[[X]]$ for some unit $W_{\pi, i} \in \RRR_{(\pi)}[[X]]$ and $(\psi_{\pi, i})_0 = \pi^{t_{\pi, i}}$.
\item Write  $\Phi(c(f)) = (u, \tau_1, \ldots, \tau_s)$ and $(\psi_{\pi, i}')_{\pi \in \PP, i \leq k_\pi} = (g_1, \ldots, g_t)$.
\item Compute $$V =  (f'/c(f) X^r)\left(g_1 \cdots g_t \right)^{-1} \ \Mod \, X^{N+r} R[[X]]$$ by recursion on the coefficients of $V$, noting that $V_0 = v$ is a unit in $R$.
\item Return $(uV, \tau_1, \ldots, \tau_s, X, X, \ldots, X, g_1, \ldots, g_t)$.
\end{enumerate}
\end{algorithm}

\begin{proof}
 For each $\pi \in \PP$ there exists a  prime factorization $(1, \Phi_{\pi, 1} , \ldots, \Phi_{\pi, l_\pi})$ of $f/c(f)X^r$ in $\RRR_{(\pi)}[X]$ with $\Phi_{\pi, i} \equiv \Psi_{\pi, i} \ (\Mod \, (f_r/c(f), X)^N)$ for all $i \leq l_\pi$. 
Let $i \leq k_\pi$.  Then $\Psi_{\pi, i} \equiv \Phi_{\pi, i} \ (\Mod \, ((\Phi_{\pi, i})_0,X)^N)$, which in turn implies $(\Psi_{\pi, i})_0 \equiv (\Phi_{\pi, i})_0 \ (\Mod \, ((\Phi_{\pi, i})_0)^2)$  and therefore $(\Psi_{\pi, i})_0 \RRR_{(\pi)}  =  (\Phi_{\pi, i})_0 \RRR_{(\pi)}$.  By Lemma \ref{preplemma}(3), then, it follows that $$\psi_{\pi,i}' = \psi_{\pi, i} \, \Mod \, X^N R[[X]] = \phi_{\pi, i} \, \Mod \, X^NR[[X]],$$
where  $\phi_{\pi, i } = V_{\pi, i} \Phi_{\pi, i} \in R[[X]]$ for some unit $V_{\pi, i} \in \RRR_{(\pi)}[[X]]$.  Therefore, by the notation introduced in step (3) we may write $(\phi_{\pi, i})_{\pi \in \PP, i \leq k_\pi} = (G_1, \ldots, G_t)$ with $G_i \equiv g_i \, (\Mod \, X^N R[[X]])$ for all $i$. By Algorithm \ref{mainfactortheorem}, $$(uW, \tau_1, \ldots, \tau_s, X, X, \ldots, X, G_1, \ldots, G_t)$$ is a prime factorization of $f$ in $R[[X]]$, where $W =  (f/c(f) X^r )(G_1 \ldots G_t)^{-1}$. Moreover, one has $$V c(f)X^r g_1 \cdots g_t \equiv f' \equiv f \equiv W c(f) X^r g_1 \cdots g_t \, (\Mod X^{N+r}R[[X]]),$$
and therefore $V \equiv W \, (\Mod X^NR[[X]])$.   Therefore $$(uV, \tau_1, \ldots, \tau_s, X, X, \ldots, X, g_1, \ldots, g_t)$$ is a mod $X^N R[[X]$ prime factorization of $f$ in $R[[X]]$.
\end{proof}

\section{Irreducibility criteria}\label{sec:7}

In this section we provide an alternative (and shorter) proof of Theorem \ref{irredtest}.  Our main result in this section is the following theorem, which, together with Corollary \ref{PIDlemma}(2), yields Theorem \ref{irredtest}.

\begin{theorem}\label{maintheorem2}
Let $R$ be an integral domain and $\ppp$ a prime ideal of $R$ such that   $\RRR_\ppp$ is an integral domain containing $R$, and let $f \in R[[X]]$ be $\widehat{\ppp}$-distinguished in $\RRR_\ppp[[X]]$.   Suppose that $\RRR_\ppp/f_0 \RRR_\ppp \cong R/f_0 R$ as $R$-algebras and any divisor of $f_0$ in $\RRR_\ppp$ is associate in $\RRR_\ppp$ to a divisor of $f_0$ in $R$.  Then $f$ is irreducible in $R[[X]]$ if and only if the Weierstrass polynomial $P_f$ associated to $f$ in $\RRR_\ppp[X]$ is irreducible in $\RRR_\ppp[X]$; in fact, one has the following.
\begin{enumerate}
\item If $f = gh$ for nonunits $g, h \in R[[X]]$, then $g$ and $h$ are $\widehat{\ppp}$-distinguished in $\RRR_\ppp[[X]]$, one has $P_f = P_g P_h$,  and $P_g$ and $P_h$ are nonunits in $\RRR_\ppp[X]$.
\item If $P_f = GH$ for nonunits $G, H$ in $\RRR_\ppp[X]$, then $f = gh$  for nonunits $g, h \in R[[X]]$ with $G = uP_g$ and $H = u^{-1} P_h$ for a unique unit $u \in \RRR_\ppp$.
\end{enumerate}
\end{theorem}

\begin{remark}
The hypotheses of Theorem \ref{maintheorem2} hold if $R$ is any integral domain, $\ppp = (\pi)$ a principal prime ideal of $R$ such that $\RRR_{\ppp}$ is an integral domain containing $R$, and $f \in R[[X]]$ is any $\widehat{\ppp}$-distinguished series in $\RRR_\ppp[[X]]$ with $f_0$ associate in $R$ to a power of $\pi$.
\end{remark}

To prove Theorem \ref{maintheorem2} we use Lemmas \ref{nzdlemma},  \ref{preplemmab}, \ref{preplemma}, and the following lemma.

\begin{lemma}\label{maintheorem3}
Let $R$ be an integral domain and $\ppp$ a prime ideal of $R$ such that   $\RRR_\ppp$ is an integral domain containing $R$.   For any $f \in R[[X]]$ that is $\widehat{\ppp}$-distinguished in $\RRR_\ppp[[X]]$, one has the following.
\begin{enumerate}
\item If $f = gh$ with $g, h \in R[[X]]$, then $g$ and $h$ are $\widehat{\ppp}$-distinguished in $\RRR_\ppp[[X]]$, and one has $P_f = P_g P_h$ and $U_f = U_g U_h$.
\end{enumerate}
Moreover, if every nonunit divisor of $f_0$ in $R$ is a nonunit in $\RRR_\ppp$, then we have the following.
\begin{enumerate}
\item[(2)] If $f = gh$ for nonunits $g, h$ in $R[[X]]$, then $P_g$ and $P_h$ are nonunits in $\RRR_\ppp[X]$.
\item[(3)] If $P_f$ is irreducible in $\RRR_\ppp[X]$, then $f$ is irreducible in $R[[X]]$.
\end{enumerate}
\end{lemma}

\begin{proof}
Suppose that $f = gh$, where $g, h \in R[[X]]$.  Since $\widehat{\ppp}$ is prime, the series $g$ and $h$ are $\widehat{\ppp}$-distinguished by Lemma \ref{preplemmab}, and we have $$U_f P_f = f = gh = (U_g P_g )(U_h P_h) = (U_g U_h)(P_g P_h),$$ and $U_g U_h$ is a unit in $\RRR_\ppp[[X]]$ and $P_g P_h$ is a $\widehat{\ppp}$-Weierstrass polynomial in $\RRR_\ppp[X]$.  By uniqueness it follows that $U_f = U_g U_h$ and $P_f = P_g P_h$.   This proves (1).

Suppose now that every nonunit divisor of $f_0$ in $R$ is a nonunit in $\RRR_\ppp$ and that $f = gh$, where $g$ and $h$ are nonunits in $R[[X]]$. Since $P_g$ is monic, if it is constant then $P_g = 1$ and so $g = U_g$, whence $g_0$ is a unit in $\RRR_\ppp$ and a nonunit divisor of $f_0$ in $R$, contradicting the hypothesis on $f_0$.  Therefore $P_g$, and likewise $P_h$, is nonconstant, hence a nonunit in $\RRR_\ppp[X]$, and $P_f = P_g P_h$ is reducible in $\RRR_\ppp[X]$.  This proves (2) and (3).
\end{proof}

\begin{proof}[Proof of Theorem \ref{maintheorem2}]
If $d$ is any divisor of $f_0$ in $R$, then $\RRR_\ppp/d \RRR_\ppp \cong R/d R$, so $d$ is a nonunit in $\RRR_\ppp$ if and only if $d$ is a nonunit in $R$.  Therefore statement (1) holds by Lemma \ref{maintheorem3}.  Suppose that $P_f = GH$, where $G, H$ are nonunits in $\RRR_\ppp[X]$.  Since $P_f$ is monic, there is a unique unit $u \in \RRR_\ppp$ such that $G' = u^{-1}G$ and $H' = u H$ are monic.   By Lemma \ref{preplemmab} it follows that $G'$ and $H'$ are $\widehat{\ppp}$-Weierstrass polynomials, say, of degree $a$ and $b$, respectively.  If $a = 0$ then $G' = 1$ and $G$ is a unit in $\RRR_\ppp[X]$, which is a contradiction.  Therefore $a$, and likewise $b$, is a positive integer.  Now, $G'_0$ and $H_0'$ are divisors of $f_0$ in $\RRR_\ppp$.  Therefore they are associate, respectively, to divisors $r$ and $s$ of $f_0$ in $R$.   Since $\RRR_\ppp /r \RRR_\ppp \cong R/rR$, and likewise for $s$, by Lemma \ref{preplemma} there exist $g, h \in R[[X]]$ with $g_0 = r$ and $h_0 = s$ and units $U,V \in \RRR_\ppp[[X]]$ such that $g = UG'$ and $h = VH'$.  Then $g$ and $h$ are $\widehat{\ppp}$-distinguished of degree $a$ and $b$, respectively, with $P_g = G'$ and $P_h = H'$.  Therefore $G = uP_g$ and $H = u^{-1} P_h$.  Moreover, we have
$f = U_f P_f = U_f G'H' = W gh$, where $g, h$ are nonunits in $R[[X]]$, and where $W = U_f U^{-1} V^{-1} \in \RRR_\ppp[[X]]$ is a unit.  Since $f = Wgh$, the constant term $W_0$ divides $f_0$ in $R$ and is a unit in $\RRR_\ppp$ and is therefore a unit in $R$.   Therefore $f_0 R = (gh)_0 R$, and since $R \cap f_0 \RRR_\ppp = f_0 R$, by Lemma \ref{nzdlemma} we have $W \in R[[X]]$.    Thus $W$ is a unit in $R[[X]]$.  Finally, setting $g' = Wg$, we see that $P_{g'} = P_g$ and $f = g'h$ is reducible in $R[[X]]$.
\end{proof}

\section{Towards an effective irreducibility test}\label{sec:8}

 In this section we provide an algorithm for testing the irreducibility of a formal power series in $R[[X]]$ with no nonunit constant or square divisors, where $R$ is the ring of integers in a number field with class number 1.

Let $K$ be a finite extension of $\QQ_p$, and let $|\cdot|$ denote the unique absolute value on $K$ extending the $p$-adic absolute value on $\QQ_p$.  Let $\lambda$ be a positive real number.  For all $f \in K[X]$ let $||f||_\lambda = \max_i |a_i|\lambda^i$.  Then $||\cdot||_\lambda$ is a nonarchimedean absolute value on $K[X]$.  Denote the discriminant of a polynomial $f \in K[X]$ by $\Delta(f)$.  

\begin{lemma}\label{irredapprox}
Let $R = \mathcal{O}_K$ be the ring of integers in a finite extension $K$ of $\QQ_p$ with uniformizing parameter $\pi \in R$, let $f$ be a polynomial in $R[X]$ of degree $n$ with $f_0 \neq 0$ and $|f_n| = 1$, and let $\lambda = |f_0|^{1/n}$.  If $g$ is any polynomial in $R[X]$ of degree $n$ with $||f-g||_\lambda < \min(|f_0|, |\Delta(f)|^2/|f_0|^{4n-3})$, then $f$ is irreducible in $R[X]$ if and only if $g$ is irreducible in $R[X]$.
\end{lemma}

\begin{proof}
Note first that $|f_0 - g_0|  \leq ||f-g||_\lambda < |f_0|$, so $|f_0| = |g_0|$ and $\lambda = |g_0|^{1/n}$.  Let $F = f/f_0$ and $G = g/f_0$.   By hypothesis one has
$$||F-G||_\lambda = |1/f_0| \, ||f-g||_\lambda < \min(1, |\Delta(f)|^2/|f_0|^{4n-4}) = \min(1, |\Delta(F)|^2).$$
Suppose that $g$ is irreducible in $R[X]$.  Then $G$ is irreducible in $K[X]$, so the Newton diagram of $G$ is a line of slope $-r = -v(g_0)/n \leq 0$, and one has $\lambda = |\pi|^r$.  It then follows from the remark after Definition 4.1 of \cite{can} that $||G||_\lambda = |G_0| = |g_0/f_0| = 1$.  Moreover, since $||F-G||_\lambda < 1 = ||G||_\lambda$, one has $||F||_\lambda = 1$, and by \cite[Lemma 8.18]{can}, the Newton diagram of $F$ is the same as that of $G$.  Therefore, since $G$ is irreducible in $K[X]$, by \cite[Corollary 8.19]{can} the polynomial $F$ is irreducible in $K[X]$, so $f = f_0 F$ is irreducible in $R[X]$ since $f$ is primitive, as $|f_n| = 1$.

Conversely, suppose that $f$ is irreducible in $R[X]$.  As above it follows that $||F||_\lambda = ||G||_\lambda = 1$ and the Newton diagram of both $F$ and $G$ is a line of slope $-r$, where $\lambda = |\pi|^r$.  By the proof of \cite[Corollary 8.19]{can} one has $|\Delta(F)| = |\Delta(G)|$, and therefore $$||F-G||_\lambda <  \min(1, |\Delta(G)|^2).$$   Therefore, again by \cite[Corollary 8.19]{can}, the polynomial $G$ is irreducible in $K[X]$.  Finally, one has $|f_n - g_n||f_0| = |f_n - g_n|\lambda^n < |f_0|$, whence $|f_n -g_n| < 1$ and so $|g_n| = |f_n| = 1$ and therefore the polynomial $g = f_0 G$ is primitive.  Thus $g$ is irreducible in $R[X]$.
\end{proof}

Let $R$ be a ring and $\aaa$ an ideal of $R$.  For any $f \in R[[X]]$ that is $\widehat{\aaa}$-distinguished in $\RRR_\aaa[[X]]$ we define $\Delta_\aaa(f)$ to be the discriminant $\Delta(P) \in \RRR_\aaa$ of the associated Weierstrasss polynomial $P = P_{f,\RRR_\aaa}$.   Suppose that $\mm = \aaa$ is maximal.  Although $\RRR_\mm$ may not be a complete local ring, it is necessarily a Henselian local ring, by the proof of \cite[Theorem 8.3 Hensel's lemma]{mat}.  Suppose that $\RRR_\mm$ is an integral domain. Let $\RRR_\mm^+$ be the absolute integral closure of $\RRR_\mm$ and $\widehat{\mm}^+$ its unique maximal ideal.  Let $\widetilde{R}_\mm = \widehat{(\RRR_\mm^+)}_{\widehat{\mm}^+}$ denote the completion of $\RRR_\mm^+$ with respect to $\widehat{\mm}^+$ and $\widetilde{\mm}$ its unique maximal ideal, and suppose that $\widetilde{R}_\mm$ is also an integral domain. Since $U = U_{f, \RRR_\mm}$ converges in $\widetilde{\mm}$ and has no zeros in $\widetilde{\mm}$, the roots of $P$ in $\RRR_\mm^+$ coincide with the zeros of $f = UP$ in $\widetilde{\mm}$.  If these roots are $\alpha_1, \ldots, \alpha_n$ (counting multiplicities) then one has $\Delta_\mm(f) = \prod_{i < j} (\alpha_i - \alpha_j)^2$.  In particular, $\Delta_\mm(f) \in \RRR_\mm$ can be computed from the zeros of $f$ (or approximated from approximate zeros of $f$) in $\widetilde{R}_\mm$, without a priori knowledge of the Weierstrass polynomial $P_{f, \RRR_\ppp}$, assuming that $\RRR_\mm$ and $\widetilde{R}_\mm$ are integral domains.

The following algorithm avails itself of standard algorithms in algebraic number theory.

\begin{algorithm}\label{irredtestalg2}
Let $K$ be a number field with class number $1$ and ring of integers $R = \mathcal{O}_K$. Let $f \in R[[X]]$, and suppose that $f$ has no nonunit constant or square factors and $f_0$ is a nonzero nonunit of $R$.   Let $\pi$ be a prime in $R$ dividing $f_0$, and let $D$ be a positive lower bound for $|\Delta_{(\pi)}(f)|$, where $|\cdot|$ is the unique absolute value on $\RRR_{(\pi)}$ with $|p| = 1/p$, where $p$ is the characteristic of the residue field $R/(\pi)$.  Let $n$ be the least nonnegative integer so that $\pi \nmid f_n$.    Let $B = \min(|f_0|, D^2/|f_0|^{4n-3})$, and choose a positive integer $N$ so that $|\pi|^N < B|\pi|^n$.  Let  $g = f \, \Mod \, (X^{(n+1)N})$.  Given the data $(R, g,\pi, n, N)$, the following algorithm determines whether or not $f$ is irreducible in $R[[X]]$.
\begin{enumerate} 
\item If $g_0 = 0$, then return IRREDUCIBLE if $g_1$ is a unit in $R$ and return REDUCIBLE otherwise.
\item If $g_0$ has at least two nonassociate prime factors in $R$, then return REDUCIBLE.
\item For each $i \leq (n+1)N$, compute a system $\Pi_i \subseteq R$ of representatives for $R/(\pi^i)$ with $0 \in \Pi_i$.
\item Using Algorithm \ref{walg}, compute $Q = P_g \, \Mod \, (\pi, X)^N \in R[X]$, where $P_g$ is the Weierstrass polynomial in $\RRR_{(\pi)}[X]$ associated to $g$.
\item Test whether or not $Q$ is irreducible in $\RRR_{(\pi)}[X]$ using known irreducibility tests, as in \cite{can} \cite{chi} \cite{guar} \cite{guar1} \cite{guar2} \cite{pau}.  Return IRREDUCIBLE if $Q$ is irreducible in $\RRR_{(\pi)}[X]$ and return REDUCIBLE otherwise.
\end{enumerate}
\end{algorithm}

\begin{proof}
 Since $f$ has no nonunit square factors, by Theorem \ref{maintheorem2}(2) the same is true of the polynomial $P = P_f$, and $\Delta_{(\pi)}(f) = \Delta(P)$ is therefore nonzero.  Since steps (1) and (2) are clearly justified, we may assume without loss of generality that $f_0 = g_0$ is nonzero and does not have two nonassociate prime factors.  By Algorithm \ref{walg} and the construction of $Q$ in step (4) one has $P = P_f \equiv P_g \equiv Q \ (\Mod \, (\pi,X)^N)$.  Moreover, since $0 \in \Pi_i$ for each $i$ and $P$ has degree $n < N$, it follows that $Q$ also has degree $n$.  Let $\lambda = |f_0|^{1/n}$.  One has
$$||P - Q||_\lambda = \max_{i \leq n} |P_i-Q_i|\lambda^i \leq  \max_{i \leq n} |P_i-Q_i| \leq  \max_{i \leq n} |\pi|^{N-i} < B,$$
and therefore since $|f_0| = |P_0|$ and $\Delta_{(\pi)}(f) = \Delta(P)$ one has $$||P-Q||_\lambda < B \leq \min(|P_0|, |\Delta(P)|^2/|P_0|^{4n-3}).$$
By Lemma \ref{irredapprox} it follows that $P$ is irreducible in $\RRR_\ppp[X]$ if and only if $Q$ is irreducible in $\RRR_\ppp[X]$.
Finally, by Theorem \ref{irredtest}, $P = P_f$ is irreducible in $\RRR_\ppp[X]$ if and only if $f$ is irreducible in $R[[X]]$.  This justifies step (5) of the algorithm and concludes the proof.
\end{proof}

\begin{remark}
Algorithm \ref{irredtestalg2} could be implemented if one had an effective algorithm for computing a positive lower bound for $|\Delta_{(\pi)}(f)|$.
\end{remark}

\section*{Acknowledgments}

I would like to acknowledge the work of James M.\ McDonough.  Some progress on the irreducibility and factoring problem was made as a consequence of the results in his thesis \cite{mcd} referenced herein.  I would also like to thank Jared Weinstein and the anonymous referee for their helpful comments on previous versions of this paper.

\end{document}